\documentclass[11pt,a4paper,reqno]{amsart}
\usepackage{a4wide}
\usepackage{amsfonts,amsmath,amssymb,amsthm,enumerate,color,hyperref,bbm,tabularx,ifthen,twoopt}
\usepackage{graphicx}
\usepackage[latin1]{inputenc}
\usepackage{natbib}

\newcommand{\eqdef}{=}
\newcommand{\eqsp}{\;}

\newcommand{\chunk}[3]{#1_{#2:#3}}
\def\rset{\mathbb{r}}
\def\zset{\mathbb{z}}

\newcounter{hypA}
\newenvironment{hypA}{\refstepcounter{hypA}\begin{itemize}
  \item[({\bf A\arabic{hypA}})]}{\end{itemize}}

\def\rmd{\mathrm{d}}
\def\ie{\emph{i.e.}}
\def\cd{\stackrel{d}{\longrightarrow}}
\def\eg{\emph{e.g.}}
\def\PE{\mathbb{E}}

\def\PP{\mathbb{P}}

\newcommand{\SkhSpace}[1]{\mathsf{D}(#1)}
\newcommand{\supnorm}[1]{\left\| #1 \right\|_\infty}

\def\Cov{\mathrm{Cov}}

\def\Var{\mathop{\rm Var}\nolimits}

\def\rmd{\mathrm{d}}

\newcommandtwoopt{\QRC}[4][][]
{\ifthenelse{\equal{#1}{}}
{\ifthenelse{\equal{#2}{}}{\mathrm{Q}_{#3}\left( #4\right)}{\mathrm{Q}_{#3}\left(#4,#2\right)}}
{\ifthenelse{\equal{#2}{}}{\mathrm{Q}^{#1}_{#3}\left(#4\right)}{\mathrm{Q}^{#1}_{#3}\left(#4,#2\right)}
}
}
\def\IF{\mathrm{IF}}
\newtheorem{theo}{Theorem}
\newtheorem{prop}[theo]{Proposition}
\newtheorem{lemma}[theo]{Lemma}

\theoremstyle{remark}
\newtheorem{remark}{Remark}

\def\rset{\mathbb{R}}
\def\zset{\mathbb{Z}}

\def \1{\mathbbm{1}}
\def\B{\textrm{B}}

\title[Robust estimation of the autocovariance function]
{Robust estimation of the scale and of the autocovariance function of
  Gaussian short and long-range dependent processes}
\author{C. L\'evy-Leduc}
\author{H. Boistard}
\author{E. Moulines}
\author{M. S.  Taqqu}
\author{V. A. Reisen}
\address{CNRS/LTCI/Telecom ParisTech - 46, rue Barrault, 75634 Paris C\'edex 13, France.}
\email{celine.levy-leduc@telecom-paristech.fr}
\address{GREMAQ, Universit\'e Toulouse 1 - Manufacture des Tabacs,
  b\^at. F, aile J.J. Laffont - 21 all\'ee de Brienne - 31000
  Toulouse}
\email{helene@boistard.fr}
\address{Institut Telecom/Telecom ParisTech - 46, rue Barrault, 75634 Paris C\'edex 13, France.}
\email{eric.moulines@telecom-paristech.fr}
\address{Department of Mathematics, Boston University, 111 Cumington
  Street, Boston, MA 02215, USA}
\email{murad@bu.edu}
\address{Departamento de Estat\'istica, Universidade Federal do
  Esp\'irito Santo, Vit\'oria/ES, Brazil}
\email{valderio@cce.ufes.br}
\date{\today}

\begin{document}

\maketitle

\begin{abstract}
A desirable property of an  autocovariance estimator is to be robust
to the presence of  additive outliers. It is well-known that the sample
autocovariance, being based on moments, does not have this property.
Hence, the use of an autocovariance estimator
which is robust to additive outliers can be very useful for time-series modeling.
In this paper, the asymptotic properties of the robust scale and
autocovariance estimators proposed by \cite{Rousseeuw:Croux:1993} and
\cite{Genton:Ma:2000} are established for Gaussian processes, with either
short-range or long-range dependence. It is shown in the short-range dependence setting that this robust
estimator is asymptotically normal at the rate $\sqrt{n}$, where $n$ is the number of observations. An explicit
expression of the asymptotic variance is also given and compared to the asymptotic variance of the classical autocovariance
estimator. In the long-range dependence setting, the limiting
distribution displays the same behavior than that of the classical
autocovariance estimator, with a Gaussian limit and rate $\sqrt{n}$ 
when the Hurst parameter $H$ is less $3/4$ and with a non-Gaussian
limit (belonging to the
second Wiener chaos) with rate depending on the Hurst parameter when $H \in (3/4,1)$.
Some Monte-Carlo experiments are presented to illustrate our claims and the Nile River data is analyzed
as an application. The theoretical results and the empirical evidence strongly suggest
using the robust estimators as an alternative to estimate
the dependence structure of Gaussian processes.
\end{abstract}

\section{Introduction}

The autocovariance function of a stationary process  plays a key role in time series analysis.
However, it is well known that the classical sample autocovariance function is very sensitive
to the presence of additive outliers in the data.  A small fraction of additive outliers,
in some cases even a single outlier, can affect the classical autocovariance estimate making it virtually useless;
see for instance \cite{deutsch:1990}
\cite{chan:1992}, \cite{chan:1995} \cite[Chapter 8]{maronna:martin:yohai:2006} and the
references therein. Since additive outliers are quite
common in practice, the definition of an autocovariance estimator which is robust to the presence of additive outliers 
is an important task.

\cite{Genton:Ma:2000} proposed a robust  estimator
of the autocovariance function and discussed
its performance on synthetic and real data sets. This estimator has later
been used by \cite{fajardo:reisen:cribari:2009} to derive
robust estimators for ARMA and ARFIMA models.

The autocovariance estimator proposed by \cite{Genton:Ma:2000} is
based on a method due to \cite{gnanadesikan:kettenring:1972}, which
consists in estimating the covariance of the random variables $X$ and $X'$
by comparing the scale of two appropriately chosen linear
combinations of these variables; more precisely, if $a$ and $b$ are non-zero, then
\begin{equation}\label{eq:cov}
\Cov(X,X')= \frac{1}{4ab}\left\{\Var(a X+ b X')-\Var(a X- b X')\right\}\; .
\end{equation}
 Assume that $\mathrm{S}$ is a robust scale functional; we write for
 short $\mathrm{S}(X)= \mathrm{S}(F_X)$, where $F_X$ is the c.d.f of
 $X$ and assume that $\mathrm{S}$ is \emph{affine equivariant}
 in the sense that $\mathrm{S}(aX+b)= |a| \mathrm{S}(X)$. Following
 \cite{huber:1981}, if we replace in the above expression
 $\Var(\cdot)$ by $\mathrm{S}^2(\cdot)$, 
then \eqref{eq:cov} is turned into the definition of a robust alternative to the covariance
 \begin{equation}
 \label{eq:robust-covariance-from-scale}
 \mathrm{C}_{\mathrm{S}} (X,X') = \frac{1}{4ab} \left\{ \mathrm{S}^2(aX+bX') - \mathrm{S}^2(aX-bX')\right\} \eqsp.
 \end{equation}
 The constants $a$ and $b$ can be chosen arbitrarily.
If $X$ and $X'$ have the same scale (\eg\ the same marginal
distribution), one could simply take $a=b=1$.
 \cite{gnanadesikan:kettenring:1972} suggest to take $a$ and $b$
 proportional to the inverse of $\mathrm{S}(X)$ and $\mathrm{S}(X')$,
 respectively in order to standardize $X$ and $X'$. 
As explained in \cite{huber:1981}, if $\mathrm{S}$ is standardized
such that $\mathrm{S}(X)=1$ 
in the case where $X$ is standard Gaussian, then, provided that
$(X,X')$ is bivariate normal,
\begin{equation}\label{e:CSC}
\mathrm{C}_{\mathrm{S}}(X,X')= \Cov(X,X')\;. 
\end{equation}
In this case indeed, 
$aX+bX'$ and $aX-bX'$ are Gaussian random variables 
with variance $\sigma^2_{\pm}=a^2 \Var(X) \pm 2 a b \Cov(X,X') + b^2 \Var(X')$,
and so, if $Y\sim\mathcal{N}(0,1)$, then
$S(aX\pm bX')=S(\sigma_{\pm} Y)=\sigma_{\pm} S(Y)=\sigma_{\pm}$ and
$S^2(aX+bX')-S^2(aX-bX')=\sigma^2_+-\sigma^2_-=4ab \Cov(X,X')$
yielding $C_S(X,X')=\Cov(X,X').$

\cite{Genton:Ma:2000} suggested to use for $\mathrm{S}$ the robust
scale estimator  introduced in \cite{Rousseeuw:Croux:1993}.
This scale estimator is based on the Grassberger-Procaccia correlation integral, defined as
\begin{equation}\label{eq:U}
r \mapsto U(r,F_X)= \iint \1_{\{ |x-x'| \leq r \}} \rmd F_X(x) \rmd F_X(x') \eqsp,
\end{equation}
which measures the probability that two independent copies $X$ and $X'$ distributed according to $F_X$ fall at a distance smaller than $r$.
The robust scale estimator introduced in \cite[p. 1277]{Rousseeuw:Croux:1993}
defines the scale $Q(F_X)$ of a c.d.f. $F_X$ as being
proportional to the first quartile of $r \mapsto U(r,F_X)$, namely,
\begin{equation}\label{eq:def:Q}
Q(F_X) = c(F_X)  \inf \left\{ r \geq 0, U(r,F_X)\geq 1/4 \right\} \eqsp,
\end{equation}
where $c(F_X)$ is a constant depending only on the shape of the
c.d.f. $F_X$.
We see immediately that $Q(F_X)$ is affine invariant, in the sense that transforming $X$ into
$a X + b$, will multiply $Q(F_X)$ by $|a|$ .
This scale can be seen as an analog of the Gini average difference
estimator $n^{-1}(n-1)^{-1}\sum_{1\leq i\neq j\leq n}|X_i-X_j|$, where the average is replaced by
a quantile. It is worth noting that instead of measuring how far away the observations are from a central value,
$Q(F_X)$ computes a typical distance between two independent copies of
the random variable $X$, which leads to a reasonable estimation of the
scale even when the c.d.f. $F_X$ is not symmetric.

The constant $c(F_X)$ in (\ref{eq:def:Q}) is there to ensure
consistency. In the sequel, the c.d.f. $F_X$ is assumed to belong to the Gaussian
location-scale family
\begin{equation}\label{eq:loc_scale}
\{\Phi_{\mu,\sigma}(\cdot)=\Phi((\cdot-\mu)/\sigma),\ \mu\in\rset,\sigma\in\rset^*_+\}\; ,
\end{equation}
where $\Phi$ is the  c.d.f. of a standard Gaussian random variable.
The reason we focus on the Gaussian family 
is that if we want to use $Q$ as the scale $S$ in (\ref{eq:robust-covariance-from-scale}), we
will need to compute $c(F_{aX+bX'})$ and $c(F_{aX-bX'})$.
This is easily done when $(X,X')$ is a Gaussian vector. Indeed, in view of
(\ref{e:CSC}), one has
\begin{equation}\label{e:CovQ}
\Cov(X,X')=\frac14\left[Q^2(F_{X+X'})-Q^2(F_{X-X'})\right]\; ,
\end{equation}
and in particular, since by (\ref{eq:U}) and (\ref{eq:def:Q}),
$Q^2(F_{2X})=(2Q(F_X))^2$, 
\begin{equation}\label{e:VarQ}
\Var(X)=Q^2(F_X)\; .
\end{equation}
When $F_X=\Phi_{\mu,\sigma}$ we can then obtain the constant
$c(\Phi_{\mu,\sigma})$ in (\ref{eq:def:Q}) explicitly as noted by 
\cite{Rousseeuw:Croux:1993}. Since $Q(\Phi_{\mu,\sigma})=\sigma$,
(\ref{eq:def:Q}) becomes 
\begin{equation}\label{e:cr}
\sigma=Q(\Phi_{\mu,\sigma})=c(\Phi_{\mu,\sigma})\sigma r_0
\end{equation}
where $r_0$ is such that, in (\ref{eq:U}), $U(r_0,\Phi)=1/4$.
Hence for all $(\mu,\sigma) \in \rset \times \rset_+^*$,
\begin{equation}\label{eq:c(Phi)}
c(\Phi_{\mu,\sigma})= c(\Phi)=1/r_0=1/(\sqrt{2}\Phi^{-1}(5/8))=2.21914\; .
\end{equation} 

Let $(X_i)_{i \geq 1}$ be a stationary Gaussian process.
Given the observations $\chunk{X}{1}{n}= (X_1,\dots,X_n)$, the
c.d.f. of the observations may be estimated using the empirical c.d.f.
$r \mapsto F_n(r)= n^{-1} \sum_{i=1}^n \1_{ \{X_i \leq r \}}$. Plugging $F_n$ into \eqref{eq:def:Q} leads to the
following robust scale estimator
\begin{equation}
\label{eq:definition-QRC}
\QRC[][\Phi]{n}{\chunk{X}{1}{n}} = c(\Phi) \{|X_i-X_j|;\ 1 \leq i,j \leq n\}_{(k_n)}\; ,
\end{equation}
where $k_n = \lfloor n^2/4 \rfloor$.
That is, up to the multiplicative constant $c(\Phi)$, 
$\QRC[][\Phi]{n}{\chunk{X}{1}{n}}$ is the $k_n$th order statistics of the
$n^2$ distances $|X_i-X_j|$ between all the pairs of
observations. 

As mentioned by \cite{Rousseeuw:Croux:1993},  $\QRC[][\Phi]{n}{\chunk{X}{1}{n}}$
has several appealing properties: it has a simple and explicit formula
with an intuitive meaning; it has the highest possible
breakdown point (50$\%$); in addition, the associated influence function (see below) is bounded.
For a definition of these quantities, which are classical in robust statistics, see for instance \cite{huber:1981}.
The scale estimator of Rousseeuw and Croux is also attractive because it can be implemented very efficiently;
it can be computed with a time-complexity of order $O(n\log n)$ and with a storage scaling linearly  $O(n)$;
see \cite{Rousseeuw:Croux:1992} for implementation details.

Using the robust scale estimator $\QRC[][\Phi]{n}{\cdot}$ in
(\ref{eq:definition-QRC}) and the identity
(\ref{eq:robust-covariance-from-scale}) with $a=b=1$, the robust autocovariance estimator
of 
$$
\gamma(h)=\Cov(X_1,X_{h+1})=\frac14\{\Var(X_1+X_{h+1})-\Var(X_1-X_{h+1})\}
$$
is
\begin{equation}\label{def:gamma_Q_init}
\widehat{\gamma}_Q(h,\chunk{X}{1}{n},\Phi)=\frac{1}{4}\left\{\QRC[2][\Phi]{n-h}{\chunk{X}{1}{n-h}+\chunk{X}{h+1}{n}}
- \QRC[2][\Phi]{n-h}{\chunk{X}{1}{n-h}-\chunk{X}{h+1}{n}}\right\}\; \eqsp.
\end{equation}
Thus, in the sample version (\ref{def:gamma_Q_init}), the random
variable $X_1\pm X_{h+1}$
is replaced by the vector $\chunk{X}{1}{n-h}\pm\chunk{X}{h+1}{n}$ of
length $n-h$.

In this paper, we establish the asymptotic properties of $\QRC[][\Phi]{n}{\chunk{X}{1}{n}}$ and the
corresponding robust autocovariance estimator $\widehat{\gamma}_Q(h,\chunk{X}{1}{n},\Phi)$ for Gaussian
processes displaying both short-range and long-range dependence. We say that the process is short-range dependent if
the autocovariance function $\{ \gamma(k) \}_{k \in \zset}$ is
absolutely summable, $\sum_{k \in \zset} |\gamma(k)| < \infty$. We say
that it is long-range dependent if the 
autocovariance function is regularly varying at infinity with exponent
$D$, $\gamma(k)= k^{-D} L(k)$ with $0 < D < 1$ and  $L$ is a slowly varying
function, 
\ie\ $\lim_{k \to \infty} L(ak)/L(k)$ for any $a > 0$, and is positive
for large enough $k$. The exponent $D$ is related to the 
so-called \emph{Hurst coefficient} by the relation $H= 1-D/2$.
See, for more details, \cite[p. 5--38]{doukhan:oppenheim:taqqu:2003}.

The limiting distributions of these estimators are obtained by using
the functional  delta method;
see \cite{vdv:2000}. In the short memory case, the results stems
directly from the weak 
invariance principle satisfied by the empirical process $F_n$ under
mild technical assumptions. 
The rate of convergence of the robust covariance estimator is
$\sqrt{n}$ and the limiting 
distribution is Gaussian; an explicit expression of the asymptotic
variance 
is given in Theorem~\ref{theo:gamma_short}.

In the long memory case, the situation is more involved. 
When $D \geq 1/2$ (or $H \leq 3/4$), the rate of convergence is 
still $\sqrt{n}$, the limiting distribution is Gaussian and the 
asymptotic variance of the covariance estimator is the same as in 
the short-memory case. When $0 < D < 1/2$, the rate of
convergence becomes equal to $n^D/\widetilde{L}(n)$ where $\widetilde{L}$ is a 
slowly varying function defined in \eqref{eq:expression-tildeL}; the
limiting distribution is non-Gaussian and belongs to the second Wiener
Chaos; see Theorem~\ref{theo:gamma_long}. We prove
that these rates are identical to the ones of the classical
autocovariance estimators. 

The study of the
asymptotic distribution of the empirical process is not enough
to derive these results.
It is necessary to use results on the empirical version of the correlation integral
which requires extensions of the results derived for $U$-processes
under short-range dependence conditions by
\cite{borovkova:burton:dehling:2001}. For this part, we use
novel results on $U$-processes of long-memory time-series that
are developed in a companion paper \cite{boistard:levy:2009}.

The outline of the paper is as follows. In Section \ref{sec:theory},
the limiting distributions of the robust scale estimator
$\QRC[][\Phi]{n}{\chunk{X}{1}{n}}$ in the Gaussian short-range 
and long-range dependence settings are proved.
From these results, the asymptotic distribution of
$\widehat{\gamma}_Q(h,\chunk{X}{1}{n},\Phi)$ is derived. In Section
\ref{sec:exp}, some Monte-Carlo experiments are presented in order to
support our theoretical claims. The Nile River data is studied
as an application in Section \ref{sec:appli}. Section \ref{s:U-process} is
dedicated to the asymptotic properties of $U$-processes which are
useful to establish the results of Section \ref{sec:theory} in the
long-range case. Sections
\ref{sec:proofs} and \ref{sec:lemmas} detail the proofs
of the theoretical results stated in Section \ref{sec:theory}.
Some concluding remarks are provided in Section \ref{sec:conc}.\ \\

\noindent\textbf{Notation.}
For an interval  $I$ in the extended real line $[-\infty,\infty]$, we denote by $\SkhSpace{I}$ the set of all functions $z: I \to \rset$ that are right-continuous and whose limits from the left exist everywhere on $I$.
We always equip $\SkhSpace{I}$ with the uniform norm, denoted by $\supnorm{\cdot}$.
We denote by $\mathcal{M}([-\infty,\infty]) $ the set of cumulative distribution functions on $[-\infty,\infty]$
equipped with the topology of uniform convergence. 
For $U \in \SkhSpace{I}$, let $U^{-1}$ denote its generalized inverse, 
$U^{-1}(\eta) \eqdef \inf \{ r \in I, U(r)\geq \eta \}$.

The convergence in distribution in
$\left(\SkhSpace{[0,\infty]},\supnorm{\cdot}\right)$ is meant 
with respect to the $\sigma$-algebra generated by the set of open
balls. We denote by $\cd$ 
the convergence in distribution. 

We denote by $\Phi$ the c.d.f of the standard Gaussian random variable
and by $\phi$ the corresponding density function.

\section{Theoretical results} \label{sec:theory}

Define the following mappings:
\begin{eqnarray}
T_1: &\mathcal{M}([-\infty,\infty])&\to \SkhSpace{[0,\infty]}\nonumber\\
&F&\mapsto \left\{r \mapsto \int_{\rset} \int_{\rset} \1_{\{\vert x-y\vert\leq
    r\}}\rmd F(x)\rmd F(y)\right\}\; ,\label{eq:definition-T1}\\
T_2: & \SkhSpace{[0,\infty]}&\to \rset\nonumber\\
&U&\mapsto U^{-1}(1/4)\; .\label{eq:definition-T2}
\end{eqnarray}
and
\begin{eqnarray}\label{eq:definition-T0}
T_0=T_2 \circ T_1: &\mathcal{M}([-\infty,\infty]))&\to\rset\\
& F&\mapsto U^{-1}(1/4)\; .
\end{eqnarray}
Then, the scale estimator $\QRC[][\Phi]{n}{\chunk{X}{1}{n}}$ introduced in
(\ref{eq:definition-QRC}) may be expressed as
\begin{equation}
\label{eq:definition-Qn}
\QRC[][\Phi]{n}{\chunk{X}{1}{n}}= c(\Phi) T_0(F_n)\; ,
\end{equation}
where $F_n$ is the empirical c.d.f. based on $\chunk{X}{1}{n}$. 

\subsection{Short-range dependence setting}

\subsubsection{Properties of the scale estimator}
The following lemma gives an asymptotic expansion for $\QRC[][\Phi]{n}{\chunk{X}{1}{n}}$,
which is used for deriving a Central Limit Theorem
(Theorem~\ref{theo:short-range}). It supposes that the empirical
c.d.f. $F_n$, adequately normalized, converges.


\begin{lemma}\label{prop:asym_expansion}
Let $(X_i)_{i\geq 1}$ be a stationary Gaussian process.
Assume that there exists a non-decreasing sequence $(a_n)$ such that
$a_n(F_n-\Phi_{\mu,\sigma})$ converges weakly in $\left(\SkhSpace{[0,\infty]},\supnorm{\cdot}\right)$.
Then, $\QRC[][\Phi]{n}{\chunk{X}{1}{n}}$ defined by \eqref{eq:definition-QRC} has
the following asymptotic expansion:
\begin{equation}\label{expansion:short-range}
a_n\left(\QRC[][\Phi]{n}{\chunk{X}{1}{n}}-
  \sigma \right)
= \frac{a_n}{n}\sum_{i=1}^n \IF(X_i,Q,\Phi_{\mu,\sigma})+ o_P(1)\; ,
\end{equation}
where, for all $x$ in $\rset$,
\begin{equation}\label{IF:Q_1}
\IF(x,Q,\Phi_{\mu,\sigma})=\sigma\IF((x-\mu)/\sigma,Q,\Phi)\; ,
\end{equation}
and
\begin{equation}\label{IF:Q}
\IF(x,Q,\Phi)
=c(\Phi)\left(\frac{1/4-\Phi(x+1/c(\Phi))
+\Phi(x-1/c(\Phi))}{\int_{\rset} \phi(y)\phi(y+1/c(\Phi))\rmd y}\right)\; .
\end{equation}

\end{lemma}
The proof of Lemma \ref{prop:asym_expansion} is given in Section \ref{sec:proofs}.
\begin{remark}
Note that $\IF(x,Q,\Phi)$ has the same expression as the influence function of the functional $Q$
evaluated at the c.d.f. $\Phi$ given by \cite[p. 1277]{Rousseeuw:Croux:1993}
and \cite[p. 675]{Genton:Ma:2000}.
As is well-known from \cite[p. 13]{huber:1981}, the
influence function $x \mapsto \IF(x,T,F)$ is defined
for a functional $T$ at a distribution $F$ at point $x$ as the limit
$$
\IF(x,T,F) \eqdef \lim_{\varepsilon\to 0+}\varepsilon^{-1}\{T(F+\varepsilon(\delta_x-F))-T(F)\}\; ,
$$
where $\delta_x$ is the Dirac distribution  at $x$.
Influence functions are a classical tool in robust statistics used to
understand the effect of a small contamination at the point $x$
on the estimator.
\end{remark} 
We focus here on the case where the process $(X_i)_{i\geq 1}$
satisfies the following assumption:

\begin{hypA}
\label{assum:short-range}
$(X_i)_{i\geq 1}$ is a stationary
mean-zero  Gaussian process with autocovariance sequence $\gamma(k)=\PE(X_{1}X_{k+1})$ satisfying:
$$
\sum_{k\geq 1}|\gamma(k)|<\infty\; .
$$
\end{hypA}

 To state the results,  we must first define the Hermite rank of the influence function $x \mapsto \IF(x,Q,\Phi)$.
 Let $\{ H_k \}$ denote the Hermite polynomials having leading coefficient equal to one.
 These are $H_0(x) = 1$, $H_1(x) = x$, $H_2(x) = x^2 - 1$, $\cdots$.
 Let $f$ be a function such that  $\int f^2(z) \rmd\Phi(z) < \infty$.
 The expansion of $f$ in Hermite polynomials is given by
 \begin{equation}
 \label{eq:edgeworth-expansion}
 f(z) \eqdef \sum_{q=\tau(f)}^\infty \frac{\alpha_q(f)}{q!} H_q(z) \eqsp,
 \end{equation}
 where $\alpha_q(f)= \int f(z) H_q(z) \rmd\Phi(z)$ and where the convergence is
 in $L^2(\rset,\Phi)$.   The index of the first nonzero coefficient in the expansion, denoted $\tau(f)$,
 is called the Hermite rank of the function $f$.  \cite[Theorem
 1]{breuer:major:1983} shows that if
 \begin{equation}
 \label{eq:convergence-autocorrelation}
 \sum_{h=-\infty}^\infty |\gamma(h)|^{\tau(f)} < \infty \eqsp,
 \end{equation}
 then the variance $\Var \left( n^{-1/2} \sum_{i=1}^n f(X_i) \right)$ converges as $n$ goes to infinity to a limiting value $\sigma^2(f)$
 which is given by
 \begin{multline}
 \sigma^2(f)=  \Var\left[f(X_1)\right] + 2 \sum_{h=1}^\infty \Cov\left[ f(X_{h+1}), f(X_1) \right]
 \\= \sum_{q=\tau}^\infty \frac{\alpha^2_q(f)}{q!} \left\{ \gamma^q(0) + 2 \sum_{h=1}^\infty \gamma^q(h)\right\} \eqsp.
 \end{multline}
 In addition, the renormalized partial sum is asymptotically Gaussian, 
 \begin{equation}
 \label{eq:convergence-loi-sum}
 n^{-1/2} \sum_{i=1}^n f(X_i) \cd \mathcal{N}\left(0,\sigma^2(f) \right) \eqsp.
 \end{equation}
 Concerning the empirical process, \cite{csorgo:mielniczuk:1996}
 proved that if
 \begin{equation}
 \label{eq:convergence-PE}
 \sum_{h=-\infty}^\infty |\gamma(h)| < \infty \eqsp,
 \end{equation}
 then $\sqrt{n}( F_{n}(\cdot) - \Phi_{0,\sigma}(\cdot))$
 converges in $\SkhSpace{[-\infty,\infty]}$ to a mean-zero Gaussian process $W(\cdot)$ with covariance
 \[
 \PE\left( W(r) W(r') \right)= \sum_{q=1}^\infty \frac{J_q(r)
   J_q(r')}{q!} \left\{ \gamma^q(0) + 2 \sum_{h=1}^\infty
   \gamma^q(h) \right\} \eqsp ,
 \]
where $J_q(r)=\int [\1_{\{\sigma x\leq r\}}-\Phi_{0,\sigma}(r)]H_q(x)\rmd\Phi(x)$ for all
$r$ in $[-\infty,\infty]$. These results are used to prove the
following theorem in Section \ref{sec:proofs}.

\begin{theo}\label{theo:short-range}
Under Assumption (A\ref{assum:short-range}), $\QRC[][\Phi]{n}{\chunk{X}{1}{n}}$ defined by
(\ref{eq:definition-QRC}), satisfies the following central limit theorem:
$$
\sqrt{n}(\QRC[][\Phi]{n}{\chunk{X}{1}{n}}-\sigma) \cd \mathcal{N}(0,\widetilde{\sigma}^2)\; ,
$$
where $\sigma \eqdef \sqrt{\gamma(0)}$ and the limiting variance $\widetilde{\sigma}^2$ is given by
\begin{equation}\label{def:sigma_tilde}
\widetilde{\sigma}^2=\gamma(0) \PE[\IF^2(X_1/\sigma,Q,\Phi)]
+2\gamma(0)\sum_{k\geq 1} \PE[\IF(X_1/\sigma,Q,\Phi) \IF(X_{k+1}/\sigma,Q,\Phi)]\; ,
\end{equation}
$\IF(\cdot,Q,\Phi)$ being defined in \eqref{IF:Q}.
\end{theo}



%
%

It is interesting to compare, under Assumption
(A\ref{assum:short-range}), the asymptotic distribution of the
 proposed estimator $\QRC[][\Phi]{n}{\chunk{X}{1}{n}}$ with that of 
the square root of the sample variance
\begin{equation}\label{eq:empirical-variance}
\widehat{\sigma}^2_{n,X} = \frac{1}{n-1}\sum_{k=1}^n (X_k - \bar{X}_n)^2 
= \frac{1}{2n(n-1)} \sum_{1 \leq i \ne j \leq n}
(X_i-X_j)^2\; ,
\end{equation} 
where $ \bar{X}_n=n^{-1}\sum_{i=1}^n X_i$.

\begin{prop}\label{p:sigma}
Under Assumption (A\ref{assum:short-range}),
$$\sqrt{n} \left(\widehat{\sigma}_{n,X} -
  \sigma\right)\stackrel{d}{\longrightarrow}
\mathcal{N}(0,\widetilde{\sigma}^2_{cl})\; ,$$
where
\begin{equation}\label{e:lvar1}
\widetilde{\sigma}^2_{cl}=(2\gamma(0))^{-1}(\gamma(0)^2+2\sum_{k\geq 1}\gamma(k)^2) \; .
\end{equation}
The relative asymptotic efficiency $\widetilde{\sigma}^2_{cl}/\widetilde{\sigma}^2$
of the estimator $\QRC[][\Phi]{n}{\chunk{X}{1}{n}}$ compared to
$\widehat{\sigma}_{n,X}$ is larger than  82.27$\%$.
\end{prop}

The index ``cl'' stands for ``classical''. 
The proof of Proposition \ref{p:sigma} is given in Section \ref{sec:proofs}.





\subsubsection{Properties of the autocovariance estimator}
In this section, we establish the limiting behavior of the autocovariance estimator given, for $0\leq h<n$, by
\begin{equation}\label{def:gamma_Q}
\widehat{\gamma}_Q(h,\chunk{X}{1}{n},\Phi)
=\frac{1}{4}\left[\QRC[2][\Phi]{n-h}{\chunk{X}{1}{n-h}+\chunk{X}{h+1}{n}}
- \QRC[2][\Phi]{n-h}{\chunk{X}{1}{n-h}-\chunk{X}{h+1}{n}}\right]\; .
\end{equation}

\begin{theo}\label{theo:gamma_short}
Assume that (A\ref{assum:short-range}) holds and let $h$ be a non negative
integer.
Then, the autocovariance estimator $\widehat{\gamma}_Q(h,\chunk{X}{1}{n},\Phi)$ satisfies the following Central Limit Theorem:
$$
\sqrt{n}\left(\widehat{\gamma}_Q(h,\chunk{X}{1}{n},\Phi)-\gamma(h)\right)
\stackrel{d}{\longrightarrow}\mathcal{N}(0,{\check{\sigma}}^2_h)\; ,
$$
where
\begin{equation}\label{e:lvar2}
\check{\sigma}^2(h)=\PE[\psi^2(X_1,X_{1+h})]+2\sum_{k\geq 1} \PE[\psi(X_1,X_{1+h}) \psi(X_{k+1},X_{k+1+h})]\; ,
\end{equation}
%
%
%
and the function $\psi$ is defined by
\begin{multline}\label{def:psi}
\psi : (x,y)\mapsto
\\ \left\{(\gamma(0)+\gamma(h)) \; \IF\left(\frac{x+y}{\sqrt{2(\gamma(0)+\gamma(h))}},Q,\Phi\right)
- (\gamma(0)-\gamma(h)) \; \IF\left(\frac{x-y}{\sqrt{2(\gamma(0)-\gamma(h))}},Q,\Phi\right)\right\}\; .
\end{multline}
where $\IF$ is defined in \eqref{IF:Q}.
\end{theo}

The proof of Theorem~\ref{theo:gamma_short} is given in
Section \ref{sec:proofs}.

\begin{remark}
Note that $\psi$ has the same expression as the influence
function of $\gamma_Q(\cdot)$ given in
\cite[p. 675]{Genton:Ma:2000}.
\end{remark}

\begin{remark}\label{rem:autocov_short}
Let us now compare under Assumption (A\ref{assum:short-range}) the
asymptotic distribution of the proposed estimator with the classical
autocovariance estimator defined by
\begin{equation}\label{e:ccov}
\widehat{\gamma}(h)=n^{-1}\sum_{i=1}^{n-h}(X_i-\bar{X}_n)(X_{i+h}-\bar{X}_n)
,\; 0\leq h\leq n-1\;. 
\end{equation}
Under (A\ref{assum:short-range}), applying \cite[Theorem
4]{Arcones:1994} to $f:(x,y)\mapsto xy$ and $X_j=(X_j,X_{j+h})$, where
$h$ is a non negative integer, leads to the following result. 

\begin{prop}
For a given
 non negative integer $h$, as $n\to\infty$,
$$\sqrt{n}(\widehat{\gamma}(h)-\gamma(h))
\stackrel{d}{\longrightarrow}\mathcal{N}(0,\check{\sigma}^2_{cl}(h))\; ,$$ 
where
\begin{equation}\label{e:lvar3}
\check{\sigma}^2_{cl}(h)=\gamma^2(0)+\gamma(h)^2+2\sum_{k\geq 1}
\gamma^2(k)+2\sum_{k\geq 1}\gamma(k+h)\gamma(k-h)\; .
\end{equation}
\end{prop}

Let us now compare $\check{\sigma}^2(h)$ in (\ref{e:lvar2}) with
$\check{\sigma}^2_{cl}(h)$ in (\ref{e:lvar3}).
Since  the theoretical lower bound for the asymptotic
relative efficiency (ARE) defined by
$\textrm{ARE}(h)=\check{\sigma}^2_{cl}(h)/\check{\sigma}^2(h)$
is difficult to obtain, the estimation of ARE was calculated in the
case where $(X_i)_{i\geq 1}$ is an
AR(1) process: $X_i=\phi_1 X_{i-1}+\varepsilon_i$, where
$(\varepsilon_i)_{i\geq 1}$ is a Gaussian white noise, for 
$\phi_1=0.1$, 0.5 and 0.9. These results are given
in Figure \ref{fig:ARE_cov} which displays ARE for $h=1,\dots,60$. 
From this figure, we can see that ARE
ranges from 0.82 to 0.90 which indicates empirically that the robust procedure
has almost no loss of efficiency.

\begin{figure}[!ht]
\begin{center}
\begin{tabular}{cc}
\includegraphics*[width=7cm]{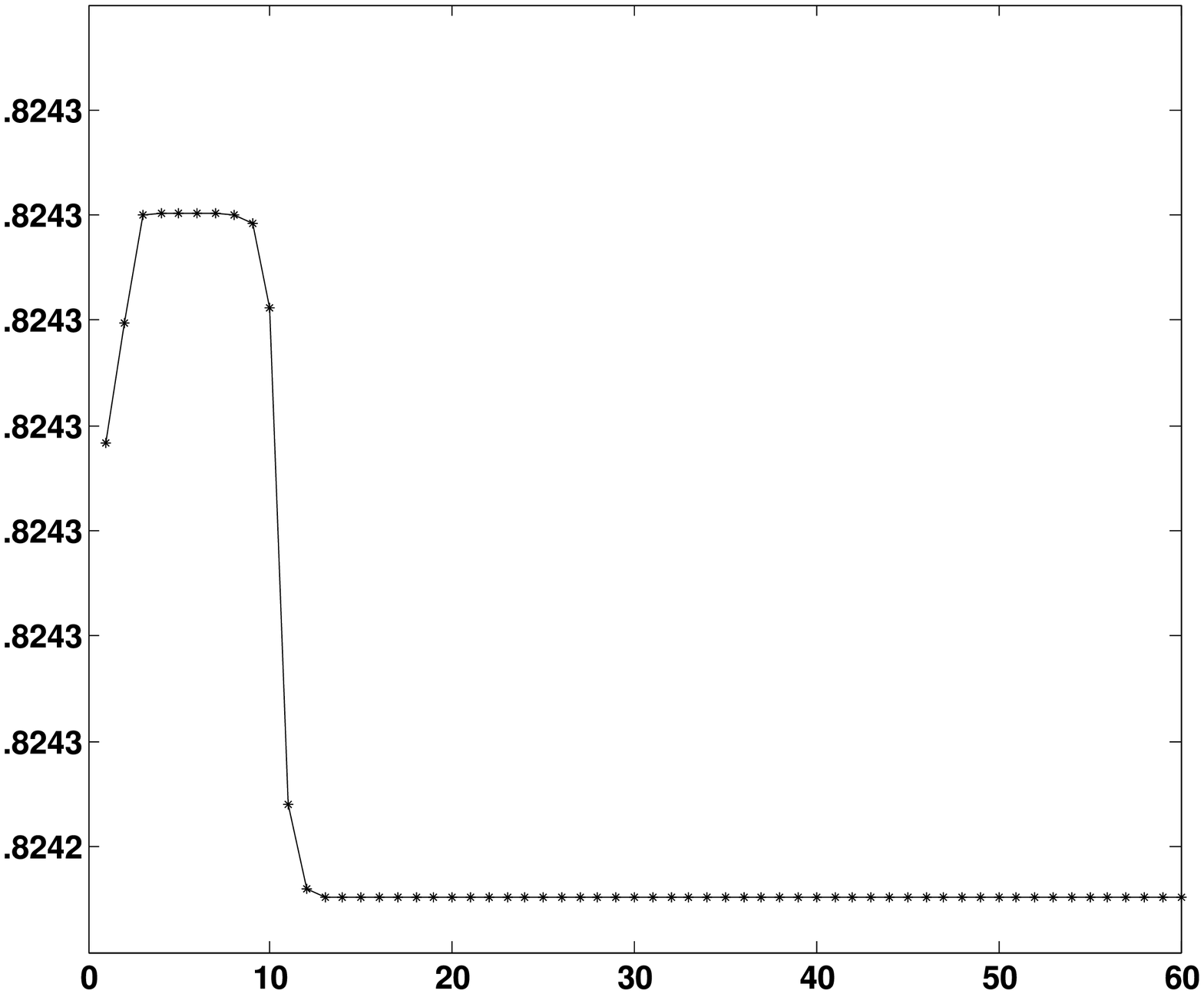}
&\includegraphics*[width=7cm]{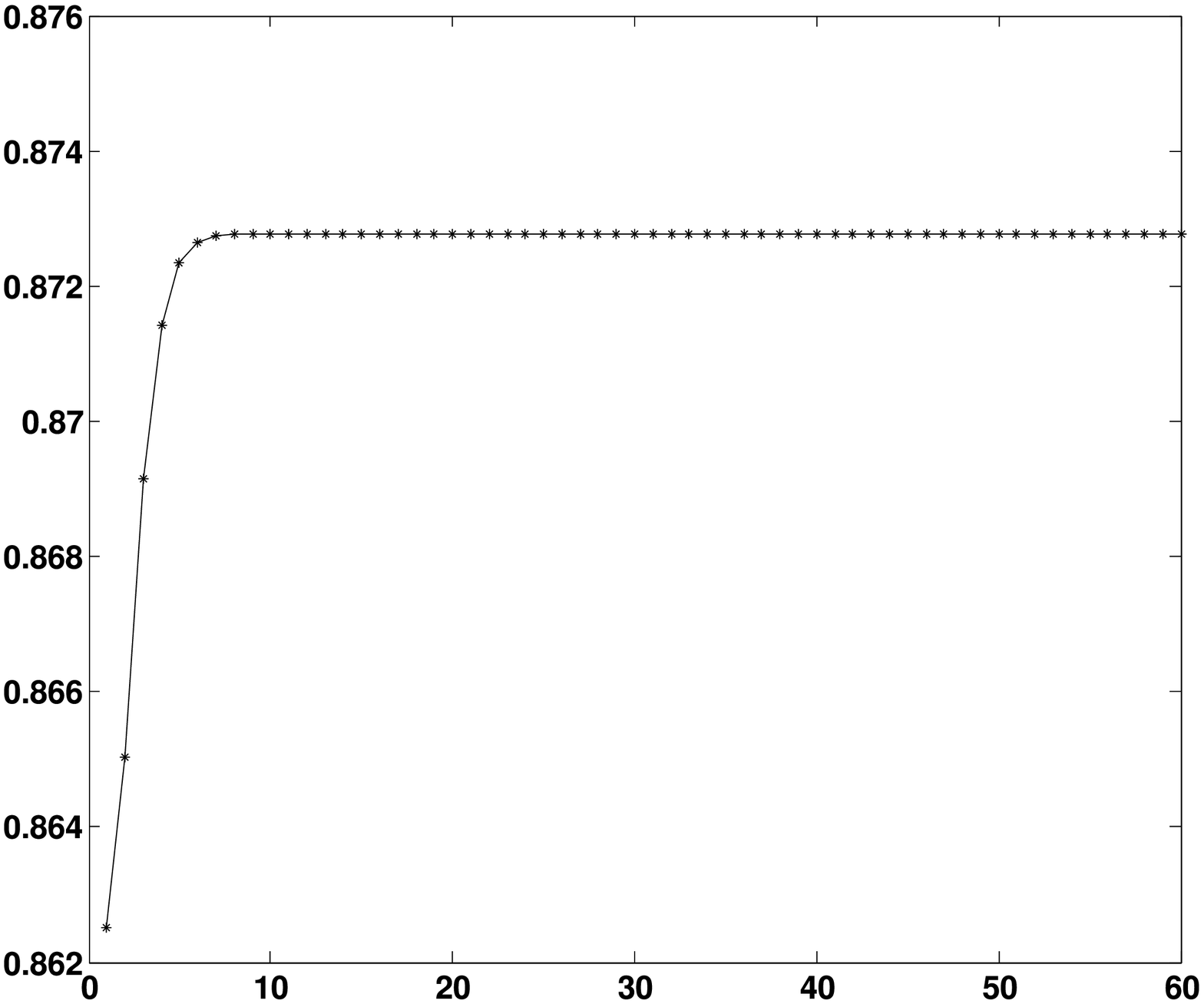}
\end{tabular}
\includegraphics*[width=7cm]{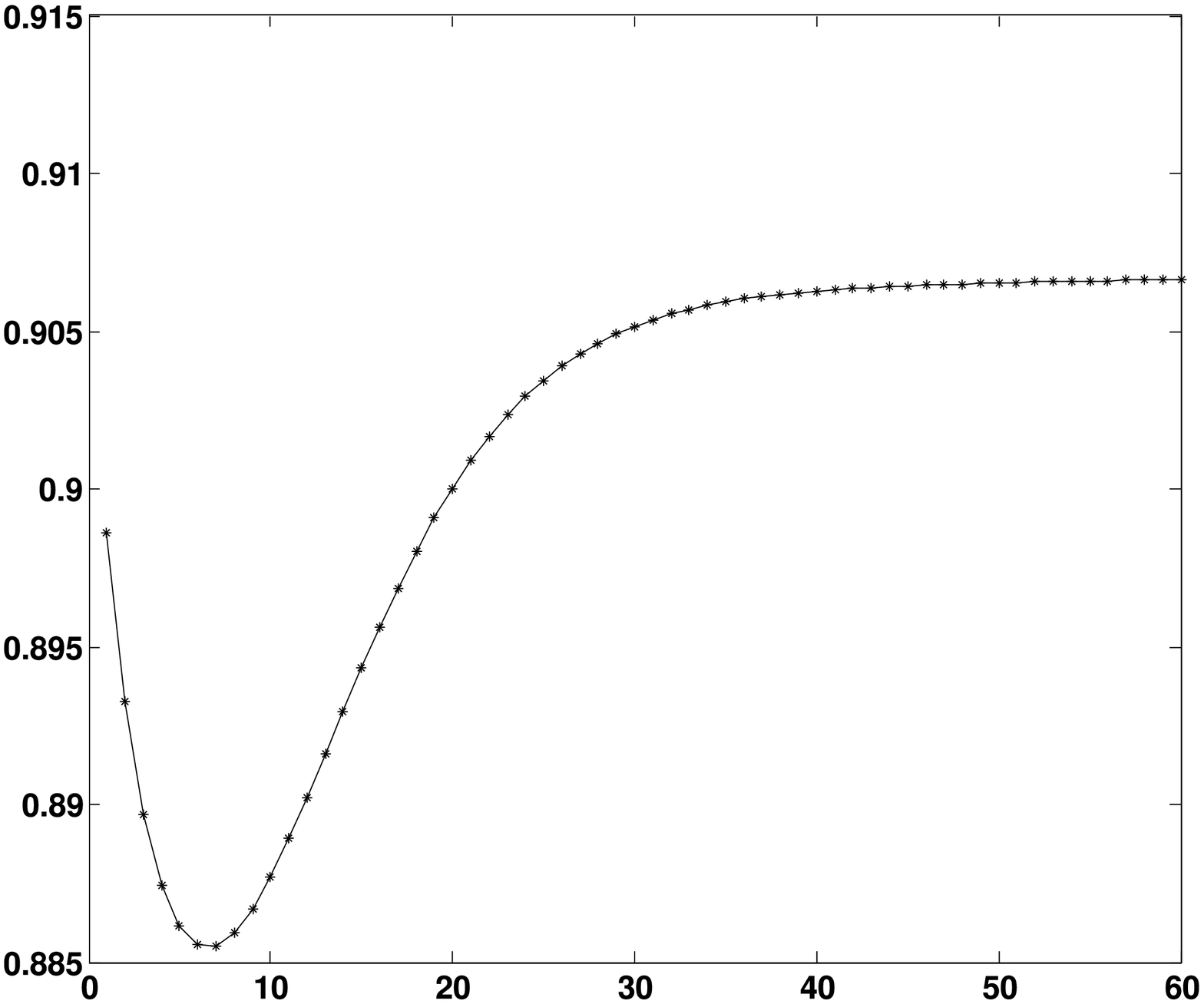}
\caption{\footnotesize{ARE for an AR(1) process for different values
    of $\phi_1$: 0.1, 0.5 and 0.9 from left to right.}}
\end{center}
\label{fig:ARE_cov}
\end{figure}

\end{remark}



\subsection{Long-range dependence setting}
In this section, we study the behavior of the robust scale and
autocovariance estimators $\QRC[][\Phi]{n}{\chunk{X}{1}{n}}$ and
$\widehat{\gamma}_Q(h,\chunk{X}{1}{n},\Phi)$
in (\ref{eq:definition-Qn}) and (\ref{def:gamma_Q}) respectively.
in the case where the
process is long-range dependent. 
Long-range dependent processes play a key role in many domains, and it is therefore worthwhile
to understand the behavior of such estimators in this context. 
\begin{hypA}
\label{assum:long-range} $(X_i)_{i \geq 1}$ is a stationary
mean-zero  Gaussian process with autocovariance $\gamma(k)=\PE(X_{1}X_{k+1})$ satisfying:
$$
\gamma(k)=k^{-D} L(k),\ 0<D<1\; ,
$$
where $L$ is slowly varying at infinity and is positive for large $k$.
\end{hypA}
A classical model for long memory process is the so-called ARFIMA($p,d,q$), which is a natural generalization of standard ARIMA($p,d,q$) models. By allowing $d$ to assume any value in $(-1/2,1/2)$, a fractional ARFIMA model is defined by $\Phi(B) (1-B)^d X_i = \Theta(B) Z_i$.
Here $(Z_i)_{i \in \zset}$ is a white Gaussian noise, $B$ denotes the
backshift operator, $\Phi(B)$ defines the AR-part, $\Theta(B)$ defines
the MA part of the process, and $(1-B)^d= \sum_{k=0}^\infty
\binom{d}{k} (-B)^k$ is the fractional difference operator. 
For $d\neq 0$, one has 
\begin{equation}\label{e:Dd}
D=1-2d
\end{equation}
(see (6.6) of \cite{taqqu:1975}).
For  $d=0$, we obtain the usual ARMA model. Long memory occurs for $d
> 0$. As $k \to \infty$, the autocovariance of an ARFIMA($p,d,q$) decreases as
$\gamma(k) = C k^{2d-1}$. Such processes satisfy
(A\ref{assum:long-range}) with $D=1-2d$, see \cite[Chapter 1]{doukhan:oppenheim:taqqu:2003}
for example for more details.

Perhaps surprisingly, the proof of the asymptotic properties
of $\QRC[][\Phi]{n}{\chunk{X}{1}{n}}$ in the long-range dependence framework does not follow the same
steps as in the short-range dependence case.

To understand why, assume that Assumption (A\ref{assum:long-range}) holds with $\gamma(0)=1$. \cite[Theorem 1.1]{Dehling:Taqqu:1989}  shows that
the difference between the empirical distribution function $F_n$ and $\Phi$, the c.d.f. of the standard Gaussian distribution $\mathcal{N}(0,1)$ renormalized by $n d_n^{-1}$, \ie\ $n d_n^{-1}(F_n- \Phi)$, converges in distribution
to a Gaussian process in the space of cadlag functions equipped with the topology of uniform convergence.
The sequence $d_n$ depends on the exponent $D$ governing the decay of the autocorrelation function to zero and also on the
slowly varying function $L$ appearing in (A\ref{assum:long-range}): more precisely,
\begin{equation}\label{e:dn}
d_n \eqdef \alpha(D)^{1/2}n^{1-D/2}L^{1/2}(n)
\end{equation}
with $\alpha(D)=2(1-D)^{-1} (2-D)^{-1}$ for $D$ in $(0,1)$
defined in (A\ref{assum:long-range}).
Therefore, Lemma~\ref{prop:asym_expansion} shows that the asymptotic
expansion of $a_n(\QRC[][\Phi]{n}{\chunk{X}{1}{n}}-1)$ in
\eqref{expansion:short-range} remains valid with $a_n=nd_n^{-1}$, and that
it remains  to study the convergence of $d_n^{-1}\sum_{i=1}^n \IF(X_i,Q,\Phi)$.
This type of non-linear functional of stationary long-memory Gaussian
sequences have been studied
in \cite{taqqu:1975} and \cite{breuer:major:1983}.
The limiting behavior of these functionals depend both on $D$ and on the Hermite rank of the function $\IF(\cdot,Q,\Phi)$.
According to \cite{breuer:major:1983} and \cite{taqqu:1975}, under Assumption
(A\ref{assum:long-range}), two
markedly different behavior may occur, depending on the value of $D$. 
If $D \in (1/2,1)$, then, by \cite{breuer:major:1983},
$n^{-1/2} \sum_{i=1}^n \IF(X_i,Q,\Phi)$ converges to a zero-mean Gaussian random variable with finite variance.
If $D \in (0,1/2)$, then  $n^{D-1} L^{-1}(n) \sum_{i=1}^n
\IF(X_i,Q,\Phi)$ converges to a non degenerate (non Gaussian) random
variable, see \cite{taqqu:1975}.
From these two results and (\ref{e:dn}), it follows that 
$$d_n^{-1}\sum_{i=1}^n
\IF(X_i,Q,\Phi)=o_P(1)\; ,$$ for $D\neq 1/2$. Therefore, the leading term in
the expansion of $\QRC[][\Phi]{n}{\chunk{X}{1}{n}}-1$ in the short-memory setting is no longer the
leading term in the long-memory case.

This explains why the proof, in the long-memory case, does not follow
the same line of reasoning as that in the short-range dependence case.
To derive the asymptotic properties of $\QRC[][\Phi]{n}{\chunk{X}{1}{n}}$ and $\widehat{\gamma}_Q(\cdot,\chunk{X}{1}{n},\Phi)$ for
long-memory processes,  it will be necessary to carry out a careful study of the $U$-process 
\begin{equation}
\label{eq:definition-U-process}
U_n(r)= \frac{1}{n(n-1)}\sum_{1\leq i\neq j\leq n}\1_{\{|X_i-X_j|\leq r\}} = T_1(F_n)[r] - \frac{1}{n} \eqsp,
\end{equation}
based on the class of kernels $ \{\1_{\{|x-y|\leq r\}} , x,y \in \rset, r\geq 0\}$.
Its asymptotic properties 
can be derived from Propositions
\ref{theo:U_n_D>1/2} and \ref{theo:D<1/2} in Section \ref{s:U-process} which are proved
in the companion paper \cite{boistard:levy:2009}.

\subsubsection{Properties of the scale estimator}

The next theorem gives the asymptotic behavior of the robust scale
estimator $\QRC[][\Phi]{n}{\chunk{X}{1}{n}}$ under Assumption (A\ref{assum:long-range}).

\begin{theo}\label{theo:Q_n_long-range}
Under Assumption (A\ref{assum:long-range}), $\QRC[][\Phi]{n}{\chunk{X}{1}{n}}$ satisfies the
following limit theorems as
$n$ tends to infinity:

\begin{enumerate}[(i)]
\item If $D>1/2$,
$$
\sqrt{n}(\QRC[][\Phi]{n}{\chunk{X}{1}{n}}-\sigma)
\stackrel{d}{\longrightarrow}\mathcal{N}(0,\widetilde{\sigma}^2)\; ,
$$
where $\sigma=\sqrt{\gamma(0)}$,
$$
\widetilde{\sigma}^2=\gamma(0)\PE[\IF(X_1/\sigma,Q,\Phi)^2]+2\gamma(0)
\sum_{k\geq 1} \PE[\IF(X_1/\sigma,Q,\Phi) \IF(X_{k+1}/\sigma,Q,\Phi)]\; ,
$$
and $\IF(\cdot,Q,\Phi)$ is defined in (\ref{IF:Q}).

\item If $D<1/2$,
$$
\beta(D)\frac{n^D}{L(n)}(\QRC[][\Phi]{n}{\chunk{X}{1}{n}}-\sigma)\stackrel{d}{\longrightarrow}
\frac{\sigma}{2}(Z_{2,D}(1)-Z_{1,D}^2(1))\; ,
$$
where $\beta(D)=\emph{B}((1-D)/2,D)$, $\emph{B}$ denoting the Beta
function
and the processes $Z_{1,D}(\cdot)$ and $Z_{2,D}(\cdot)$ are defined in (\ref{eq:fBm}) and (\ref{eq:rosenblatt}).
\end{enumerate}
\end{theo}
Theorem~\ref{theo:Q_n_long-range} is proved in
Section \ref{sec:proofs}.

\begin{remark}\label{rem:asymetry}
Note that in the case (ii) the limit distribution is not centered and
is asymmetric. Moreover, it can be proved (see
\cite{boistard:levy:2009}) that $\PE[Z_{2,D}(1)-Z_{1,D}(1)^2]=-2\beta(D)/(-D+1)(-D+2)$.
\end{remark}

\begin{remark}\label{rem:eff_scale_long}
 Under Assumption (A\ref{assum:long-range}), it is interesting to compare the asymptotic distribution of the
 proposed estimator $\QRC[][\Phi]{n}{\chunk{X}{1}{n}}$ with that of the square root of the sample variance
 $\widehat{\sigma}^2_{n,X}$ defined in (\ref{eq:empirical-variance}).


\begin{prop}\label{p:sigma2}
Suppose Assumption (A\ref{assum:long-range}). Then as $n\to\infty$,
\begin{enumerate}
\item[(a)] if $D>1/2$, 
$$\sqrt{n} \left(\widehat{\sigma}_{n,X} - \sigma
 \right)\stackrel{d}{\longrightarrow}\mathcal{N}(0,\widetilde{\sigma}^2_{cl})\; ,$$
where $\widetilde{\sigma}^2_{cl}$ is given in (\ref{e:lvar1})
\item[(b)] if $D<1/2$, 
\begin{equation}\label{e:lsigma2}
\beta(D) n^D L(n)^{-1}(\widehat{\sigma}_{n,X}-\sigma) \stackrel{d}{\longrightarrow}
\sigma/2 \left(Z_{2,D}(1) - Z^2_{1,D}(1) \right)\; .
\end{equation}
\end{enumerate}
The rates of convergence of the square
root of the sample variance $\widehat{\sigma}_{n,X}$ and of
the robust estimator $\QRC[][\Phi]{n}{\chunk{X}{1}{n}}$ are
identical. Moreover, there is no loss of efficiency when $D < 1/2$.
\end{prop}

The proof of Proposition \ref{p:sigma2} is in Section \ref{sec:proofs}.

\end{remark}

\subsubsection{Properties of the autocovariance estimator}

In this section, we study the asymptotic properties of $\widehat{\gamma}_Q(\cdot,\chunk{X}{1}{n},\Phi)$
based on the asymptotic properties of  $\QRC[][\Phi]{n}{\chunk{X}{1}{n}}$.


\begin{theo}\label{theo:gamma_long}
Assume that (A\ref{assum:long-range}) holds and that
$L$ has three continuous derivatives. Assume also that
$L_i(x)=x^i L^{(i)}(x)$ satisfy: $L_i(x)/x^{\epsilon}=O(1)$, for some
$\epsilon$ in $(0,D)$, as $x$ tends to infinity, for all $i=0,1,2,3$, where
$L^{(i)}$ denotes the $i$th derivative of $L$.
Let $h$ be a non negative integer. 
Then, $\widehat{\gamma}_Q(h,\chunk{X}{1}{n},\Phi)$ satisfies the following limit theorems as
$n$ tends to infinity.

\begin{enumerate}[(i)]
\item If $D>1/2$,
$$
\sqrt{n}\left(\widehat{\gamma}_Q(h,\chunk{X}{1}{n},\Phi)-\gamma(h)\right)
\stackrel{d}{\longrightarrow}\mathcal{N}(0,{\check{\sigma}}^2(h))\; ,
$$
where
$$
\check{\sigma}^2(h)=\PE[\psi(X_1,X_{1+h})^2]+2\sum_{k\geq 1} \PE[\psi(X_1,X_{1+h}) \psi(X_{k+1},X_{k+1+h})]\; ,
$$
$\psi$ being defined in (\ref{def:psi}).
\item If $D<1/2$,
$$
\beta(D)\frac{n^{D}}{\widetilde{L}(n)}\left(\widehat{\gamma}_Q(h,\chunk{X}{1}{n},\Phi)-\gamma(h)\right)
\stackrel{d}{\longrightarrow}
\frac{\gamma(0)+\gamma(h)}{2}(Z_{2,D}(1)-Z_{1,D}(1)^2)
$$
where  $\beta(D)=\emph{B}((1-D)/2,D)$, $\emph{B}$ denotes the Beta
function, the processes $Z_{1,D}(\cdot)$ and $Z_{2,D}(\cdot)$ are
defined in (\ref{eq:fBm}) and (\ref{eq:rosenblatt}), and
\begin{equation}
\label{eq:expression-tildeL}
\widetilde{L}(n)=2L(n)+L(n+h)(1+h/n)^{-D}+L(n-h)(1-h/n)^{-D}\eqsp .
\end{equation}
\end{enumerate}
\end{theo}

Theorem \ref{theo:gamma_long} is proved in Section \ref{sec:proofs}.

\begin{remark}
Note that the assumptions on $L_i$ made in Theorem~\ref{theo:gamma_long} are obviously satisfied if $L$ is the
logarithmic function or a power of it.
\end{remark}

\begin{prop}\label{p:sigma3}
Under Assumption (A\ref{assum:long-range}) with $D<1/2$, the robust
autocovariance estimator $\widehat{\gamma}_Q(h,\chunk{X}{1}{n},\Phi)$
has the same asymptotic behavior as the classical autocovariance
estimator (\ref{e:ccov}). There is no loss of efficiency.
\end{prop}

Proposition \ref{p:sigma3} is proved in Section \ref{sec:proofs}.



\section{Numerical experiments}\label{sec:exp}

In this section,  we investigate the robustness properties of the
estimator $\widehat{\gamma}_Q(h,\chunk{X}{1}{n},\Phi)$ in
(\ref{def:gamma_Q_init}),
using Monte Carlo experiments.

We shall regard the observations $X_t$, $ t=1,\dots,n$, as a
stationary series $Y_t$, $ t=1,\dots,n$, corrupted by additive
outliers of magnitude $\omega$. Thus we set 
\begin{align}\label{eq:ao}
X_t=Y_t+\omega W_t,
\end{align}
where $W_{t}$ are i.i.d.  random variables such that 
$\PP\left(W=-1\right)=\PP\left(W=1\right)=p/2$ and $\PP\left(W=0\right)=1-p$, where $\mathbb{E}[W]=0$ and
$\mathbb{E}[W^2]=\Var(W)=p$. Observe that $W$ is the product of Bernoulli($p$)
and \textit{Rademacher} independent random variables; the latter equals $1$ or $-1$, both with
probability $1/2$.
$(Y_t)$ is a stationary time series and it is assumed that $Y_t$ and $W_{t}$ are independent random variables.
The empirical study is based on 5000 independent replications  with
$n=100$, 500, $p=5\%, 10\%$ and $\omega= 10$.  Other cases were
also simulated, for example, series with $\omega= 3,5$ which are
magnitudes that cause less impact on the estimates compared with
$\omega= 10$. These additional results   are available upon request.

We consider first the case where $Y_t$ follows a Gaussian AR(1)
process,  that is, $Y_t=\sum_{j\geq 0}\phi_1^j Z_{t-j}$ with
$\phi_1=0.2, 0.5$ and $\{Z_t\}$ i.i.d $\mathcal{N}(0,1)$.
Then we suppose that, $Y_t$ are Gaussian
ARFIMA$(0,d,0)$ processes, that is,
\begin{equation}\label{e:FARIMA}
Y_t=(I-B)^{-d} Z_t=\sum_{j\geq 0}\frac{\Gamma(j+d)}{\Gamma(j+1)\Gamma(d)}Z_{t-j}
\end{equation}
with $d$= 0.2, 0.45 and $\{Z_t\}$ i.i.d $\mathcal{N}(0,1)$.

Classically, scale is measured by the standard deviation $\sigma$.
The robust measure of scale we consider here is $Q(F_X)$, defined
in (\ref{eq:def:Q}). Recall that one has $\sigma=Q(F_X)$ in the
Gaussian case (see \ref{e:CovQ}). We want to compare their respective
estimators $\widehat{\sigma}_{n,X}$ defined in
(\ref{eq:empirical-variance}) and $\QRC[][\Phi]{n}{\chunk{X}{1}{n}}$
defined in (\ref{eq:definition-Qn}).

The standard deviations of the AR(1) models are $Q(F_Y)=\sigma_Y = 1.0206$ and
$Q(F_Y)=\sigma_Y = 1.1547$ for $\phi_1$= 0.2 and $\phi_1$ = 0.5, respectively.
In the case of ARFIMA processes, the standard deviations are $Q(F_Y)=\sigma_Y=1.0481$ when $d=0.2$
and $Q(F_Y)=\sigma_Y=1.9085$ when $d=0.45$.
This is because the variance of AR(1) is $(1-\phi_1^2)^{-1}$ and
that of ARFIMA(0,$d$,0) is $\Gamma(1-2d)/\Gamma^2(1-d)$ (see \cite{brockwell:davis:1991}).
Figure \ref{fig:AR} and Table \ref{table} involve
AR processes, and Figures \ref{fig:ARFIMA_0.2}, \ref{fig:ARFIMA_0.45}
and \ref{fig:autocorr_simul} involve the ARFIMA processes.

\subsection{Short-range dependence case}

  Figure \ref{fig:AR} gives some insights on Theorem
  \ref{theo:short-range} and Proposition \ref{p:sigma}. In the left part of Figure
  \ref{fig:AR}, the empirical
  distribution of the quantities $\sqrt{n}(\QRC[][\Phi]{n}{\chunk{X}{1}{n}}-\sigma_Y)$ and
$\sqrt{n}(\widehat{\sigma}_n-\sigma_Y)$ are displayed. Both present
shapes close to the Gaussian density, and their standard deviations are
equal to 0.8232 and 0.7377, respectively.
These empirical standard deviations are close to
0.8233 and 0.7500 which are the values of  the asymptotic
standard deviation $\widetilde{\sigma}$ in (\ref{def:sigma_tilde})
and that of $\sqrt{n}(\widehat{\sigma}_n-\sigma_Y)$ in
(\ref{e:lvar1}), respectively.
The value 0.8233 was obtained through numerical simulations and the
value 0.7500 from the fact that for an AR(1)
$\gamma(k)=\phi_1^k(1-\phi_1^2)^{-1}$
and hence $\widetilde{\sigma}^2_{cl}=(1+2\phi_1^2)/(2(1-\phi_1^2))$ in (\ref{e:lvar1}).
Hence the empirical evidence fits with the
theoretical results of Theorem~\ref{theo:short-range} and Proposition \ref{p:sigma}.

In the right part of Figure \ref{fig:AR}, we display the results when
outliers are present.
The  empirical distribution of $\sqrt{n}(\widehat{\sigma}_n-\sigma_Y)$ is
clearly located far away
from zero. One can also observe the increase in the variance.
The quantity
$\sqrt{n}(\QRC[][\Phi]{n}{\chunk{X}{1}{n}}-\sigma_Y)$ looks symmetric
and is located close to zero.

\begin{figure}[!ht]
\begin{tabular}{cc}
\includegraphics*[width=7cm,height=5cm]{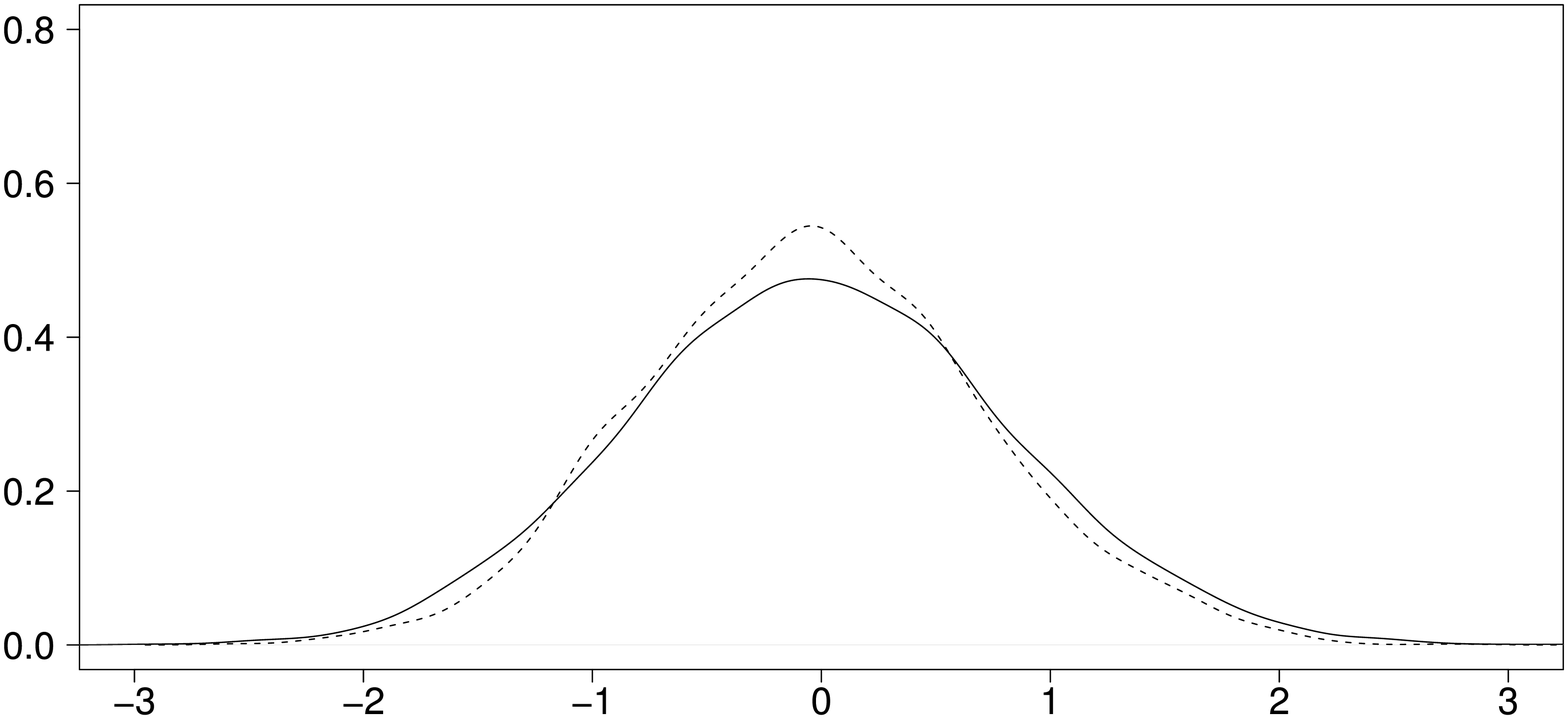}
&\includegraphics*[width=7cm,height=5cm]{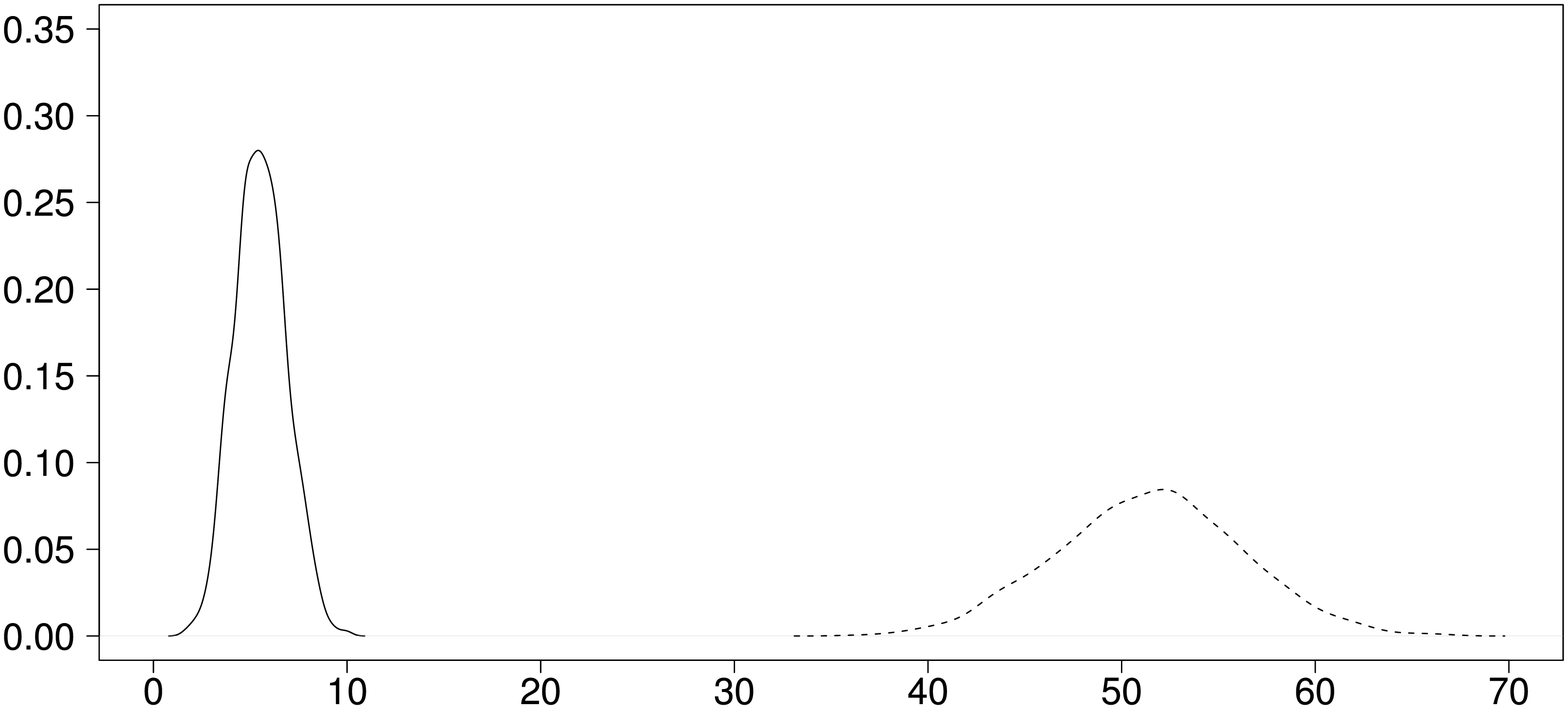}
\end{tabular}
\caption{\footnotesize{Empirical densities   of the
    quantities $\sqrt{n}(\QRC[][\Phi]{n}{\chunk{X}{1}{n}}-\sigma_Y)$
    (plain line) and
    $\sqrt{n}(\widehat{\sigma}_n-\sigma_Y)$ (dotted line) of the AR(1) model with $\phi_1
    =0.2$, $n$=500, without outliers (left) and with outliers with $p=10\%$
    and $\omega=10$ (right).}}
\label{fig:AR}
\end{figure}

We now turn to the estimation of the autocovariances. We want to use
them to get estimates for the AR(1) coefficient $\phi_1$.
The results are in Table \ref{table}.
In this table, $\widehat{\phi}_{1,\gamma}$  and $\widehat{\phi}_{1,Q}$
denote the average
of the Yule-Walker  estimates of the AR coefficients
based on the classical estimator of the covariance $\gamma$ and the
robust autocovariance
estimator $\widehat{\gamma}_Q(h,\chunk{X}{1}{n},\Phi)$ 
in (\ref{def:gamma_Q_init}), respectively. The numbers in parentheses are
the corresponding square root of the sample mean squared errors. The classical
estimates were obtained using the subroutine DARMME in FORTRAN which
uses a method of moments.  The
robust autocovariance and autocorrelation estimates were calculated
using the  code given in \cite{Rousseeuw:Croux:1992}.

\begin{table} [ht]
\centering\caption{\footnotesize{Results for the estimation of AR(1) model with
  $\omega$ =10}}
\begin{center}
\footnotesize{
\begin{tabular}{cccccccc}
  \hline\hline
\multicolumn{2}{r}{ }&\multicolumn{2}{c}{$p =0$ }&\multicolumn{2}{c}{$p =5\%$ }&\multicolumn{2}{c}{$p=10\%$ }\\
 & & & & & & & \\[-2ex]
 \cline{3-8}
 & & & & & & & \\[-2ex]
$\phi_1$ &$n$   & $\widehat{\phi}_{1,\gamma}$ &$\widehat{\phi}_{1,Q}$&
$\widehat{\phi}_{1,\gamma}$&$\widehat{\phi}_{1,Q}$&$\widehat{\phi}_{1,\gamma}$
&$\widehat{\phi}_{1,Q}$\\
  \hline
 $0.2$ & 100& 0.1818 &0.1831 &0.0312  &0.2212&0.01530&0.2651\\
       &    & (0.0112)&(0.0128)&(0.0376)&(0.0229)&(0.0435)&(0.0388)\\
  & 500& 0.1967  &  0.1948 &0.0318&0.2381&0.0163&0.2881\\
   &    & (0.0019)&(0.0025)&(0.0303)&(0.0051)&(0.0357)&(0.0150)\\
\hline
   $0.5$ & 100& 0.4767& 0.4747&0.0998  &0.5762&0.0495&0.6924\\
       &    & (0.0084)&(0.0106)&(0.1740)&(0.0262)&(0.2142)&(0.0712)\\
  & 500& 0.4967  &  0.4927 &0.1030&0.6012&0.05647&0.7216\\
   &    & (0.0015)&(0.0021)&(0.1598)&(0.0141)&(0.1988)&(0.0558)\\
\hline
\hline
\end{tabular}}
\end{center}
\label{table}
\end{table}

It can be seen from Table \ref{table} that both autocovariances yield similar
estimates for $\phi_1$ when the process does not
contain outliers. However, the picture changes
significantly when the series is contaminated by atypical
observations. As expected, the estimates from the classical
autocovariance estimator are extremely sensitive to the presence of
additive outliers. As noted in
\cite{fajardo:reisen:cribari:2009}, for a fixed lag $k$,  the
classical autocorrelation tends to zero as the weight $\omega \rightarrow
\infty$, and this produces a loss of memory  property (that is, the
dependence structure of the model is reduced),
and consequently this leads to parameter estimates with significant
negative bias.
It is worth noting that the estimator based on the robust
autocovariance (\ref{def:gamma_Q})
yields much more accurate estimates when the data contain outliers.

\subsection{Long-range dependence case}

In the case of the long-memory process ARFIMA$(0,d,0)$ defined in
(\ref{e:FARIMA}), we choose
$d=0.2$ and $d=0.45$, corresponding respectively to $D=0.6$ and
$D=0.1$ (see \ref{e:Dd}). In the first case $D>1/2$, in the second, 
$D<1/2$, corresponding to the two cases of Theorem
\ref{theo:Q_n_long-range}.
For $d=0.2$, the empirical density functions
of $\sqrt{n}(\QRC[][\Phi]{n}{\chunk{X}{1}{n}}-\sigma_Y)$ and
$\sqrt{n}(\widehat{\sigma}_n-\sigma_Y)$ are displayed in Figure
\ref{fig:ARFIMA_0.2} with and without outliers.
When there is no outlier, both shapes are similar to that of the Gaussian
density, and their standard deviations are
equal to 0.9043 and 0.8361, respectively, corresponding to an
asymptotic relative efficiency of 85.48$\%$.
As shown in the right part of Figure \ref{fig:ARFIMA_0.2},
the classical scale estimator $\widehat{\sigma}_n$ is much more sensitive
to outliers than the robust one $Q_n$. The empirical density in the
case of outliers is centered around 50.

\begin{figure}[!ht]
\begin{tabular}{cc}
\includegraphics*[width=7cm,height=5cm]{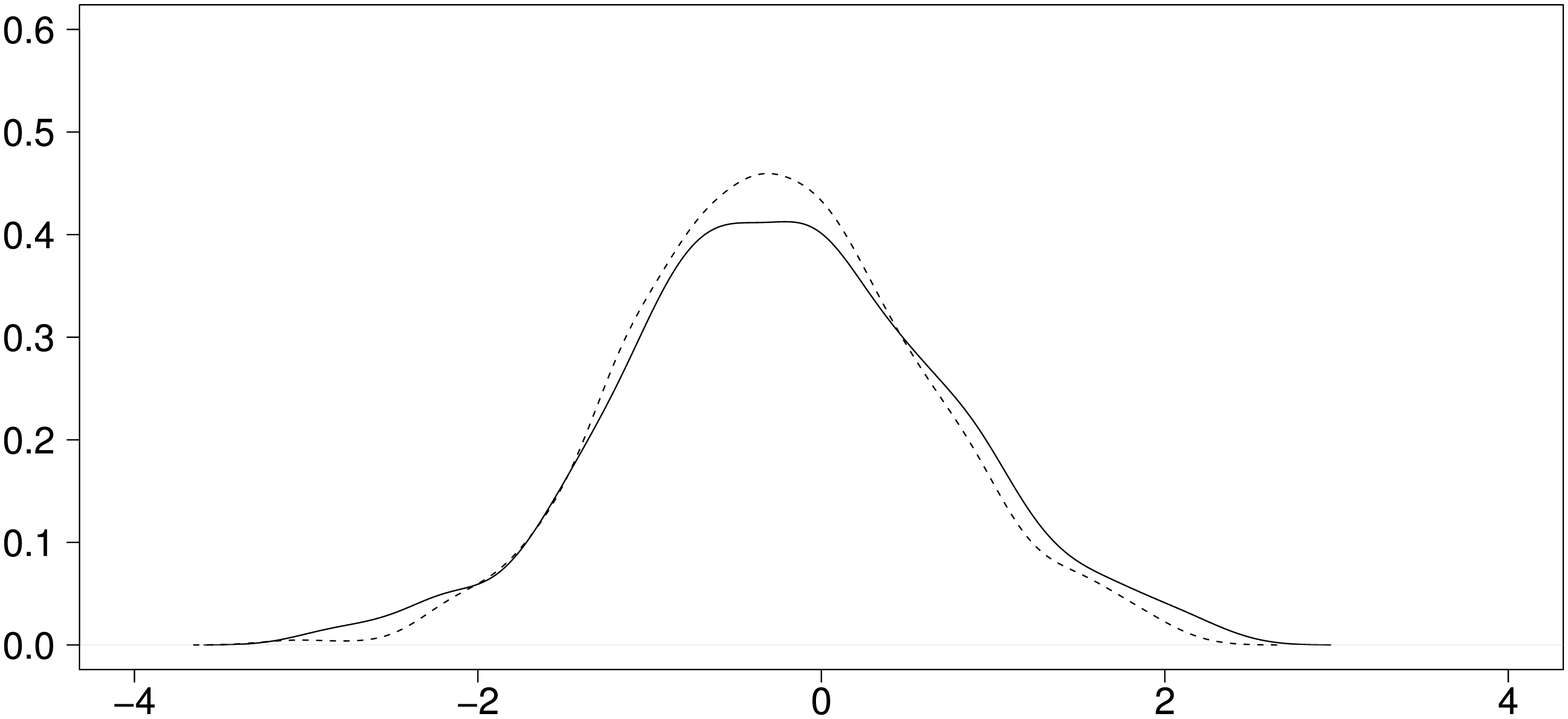}
& \includegraphics*[width=7cm,height=5cm]{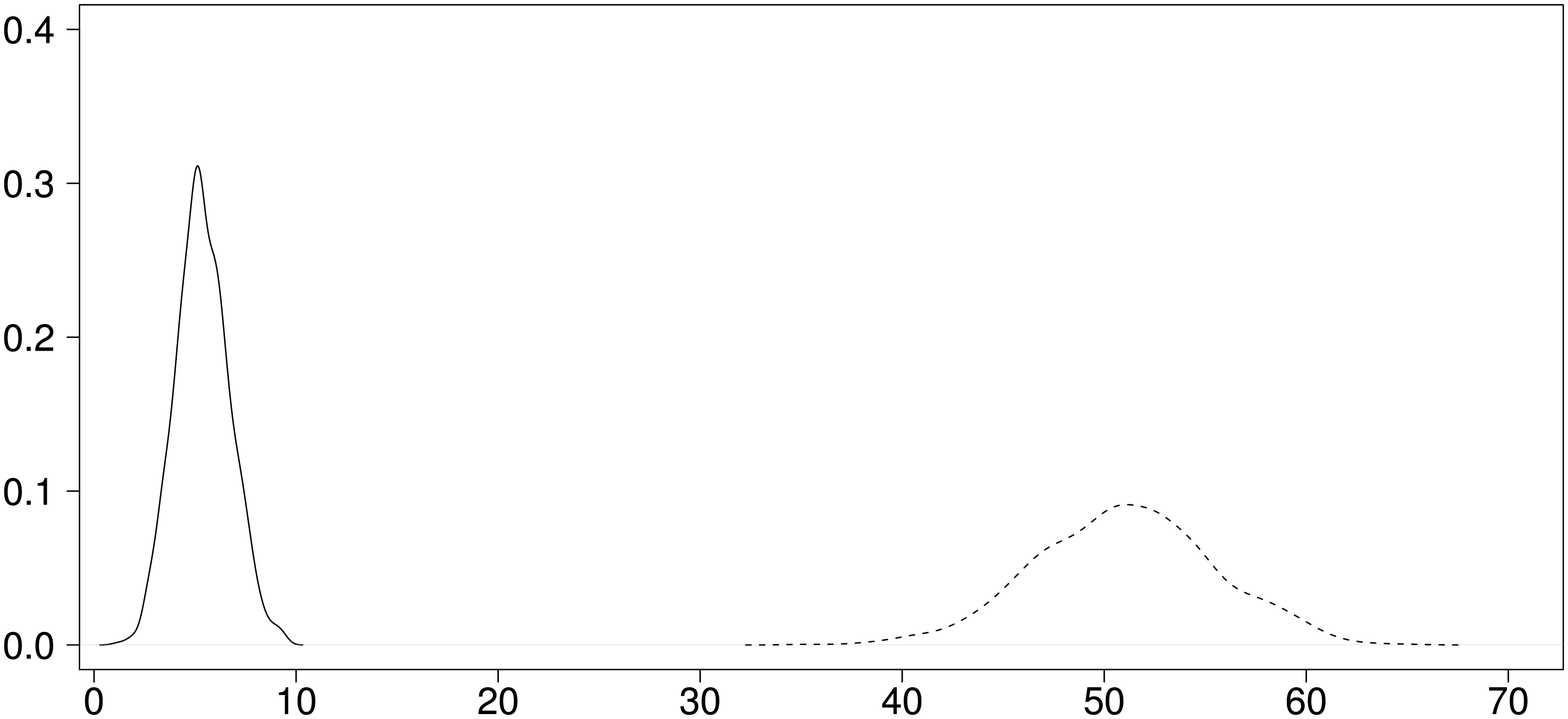}
\end{tabular}
\caption{\footnotesize{Empirical densities of
    $\sqrt{n}(\QRC[][\Phi]{n}{\chunk{X}{1}{n}}-\sigma_Y)$ (plain line) and
    $\sqrt{n}(\widehat{\sigma}_n-\sigma_Y)$ (dotted line)
    for the
    ARFIMA$(0,d,0)$ model with $d =0.2$, $n$=500 without
    outliers (left) and with outliers with $p=10\%$
    and $\omega=10$ (right).}}
\label{fig:ARFIMA_0.2}
\end{figure}

To illustrate part (ii) of Theorem~\ref{theo:Q_n_long-range},
we consider the empirical density functions 
of the quantities $n^{1-2d}(\QRC[][\Phi]{n}{\chunk{X}{1}{n}}-\sigma_Y)$ and
$n^{1-2d}(\widehat{\sigma}_n-\sigma_Y)$ when $d=0.45$ ($D=0.6$) as
displayed in Figure \ref{fig:ARFIMA_0.45}.
The left part of Figure \ref{fig:ARFIMA_0.45} shows densities
having means close to -1.1161 which is the value of the theoretical
mean given in Remark \ref{rem:asymetry}. Both curves present, in fact, similar
empirical standard deviation which is in accordance with 
Proposition \ref{p:sigma2}.
The impact of outliers
on the estimates is clearly shown in the right side of
Figure \ref{fig:ARFIMA_0.45} where one observes patterns similar to
those of the previous examples.

\begin{figure}[!ht]
\begin{tabular}{cc}
\includegraphics*[width=7cm,height=5cm]{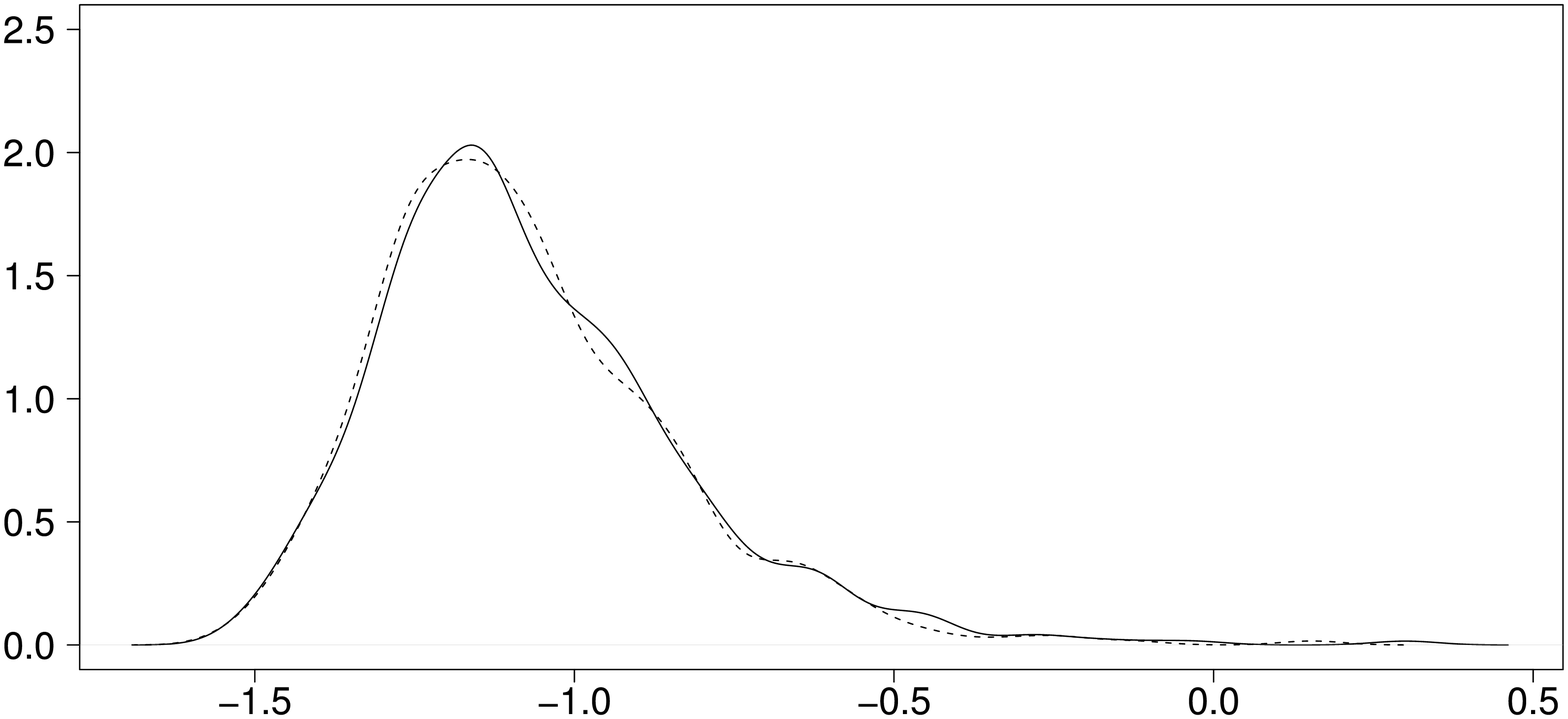}
&\includegraphics*[width=7cm,height=5cm]{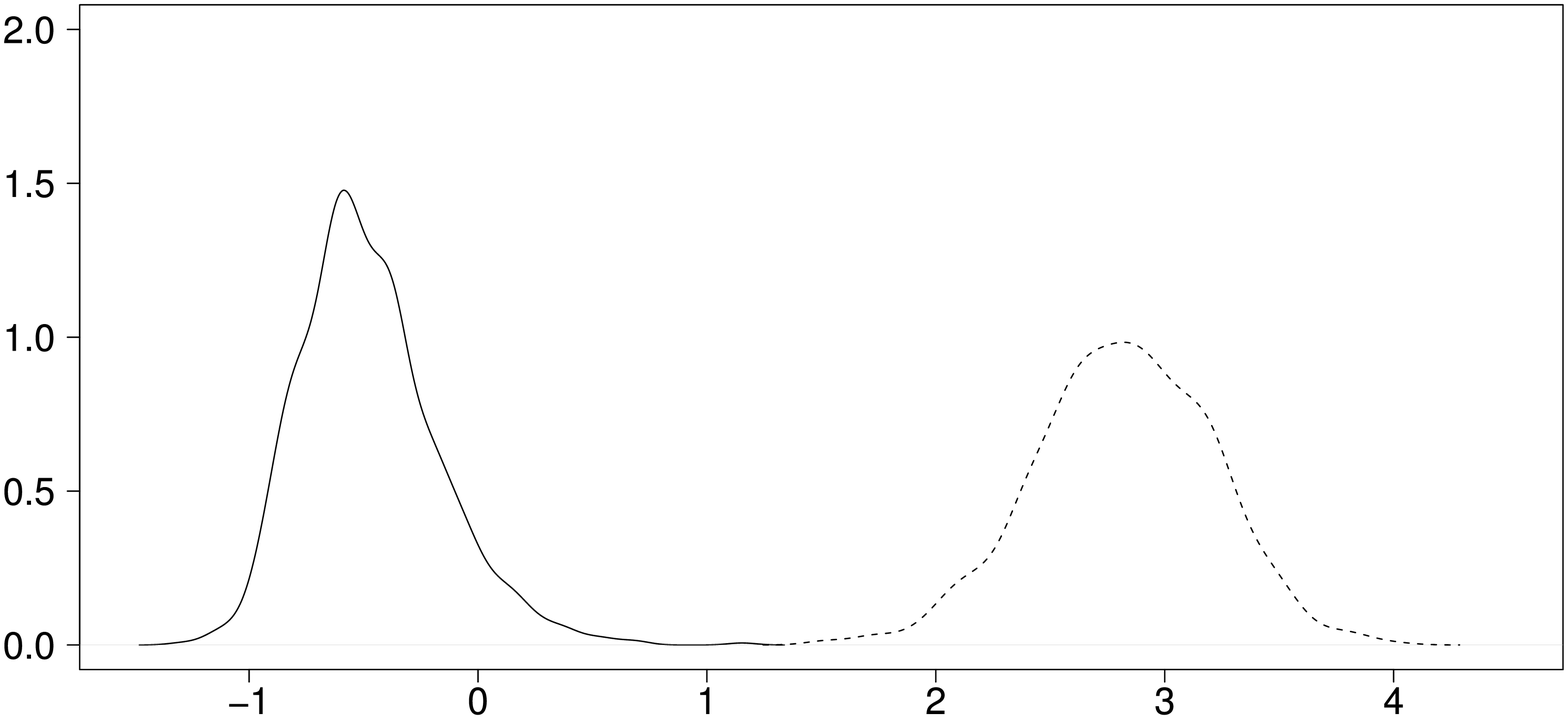}
\end{tabular}
\caption{\footnotesize{Empirical densities   of the
    quantities $n^{1-2d}(\QRC[][\Phi]{n}{\chunk{X}{1}{n}}-\sigma_Y)$
    (plain line) and
    $n^{1-2d}(\widehat{\sigma}_n-\sigma_Y)$ (dotted line) of the
    ARFIMA$(0,d,0)$ model with $d =0.45$, $n$=500, without outliers
    (left) and with outliers $p=10\%$
    and $\omega=10$ (right).}}
\label{fig:ARFIMA_0.45}
\end{figure}

Finally, the plots of the autocorrelations
 are displayed in the left and right parts of Figure \ref{fig:autocorr_simul}
 for models without and with outliers, respectively.  The figures also provide the population autocorrelation function
$\rho(h)=\Gamma(1-d)\Gamma(h+d)/(\Gamma(d)\Gamma(1+h-d))$ as a
function of the lag $h$ (\cite{hosking:1981}).

\begin{figure}[!ht]
\begin{tabular}{cc}
\includegraphics*[width=7cm,height=5cm]{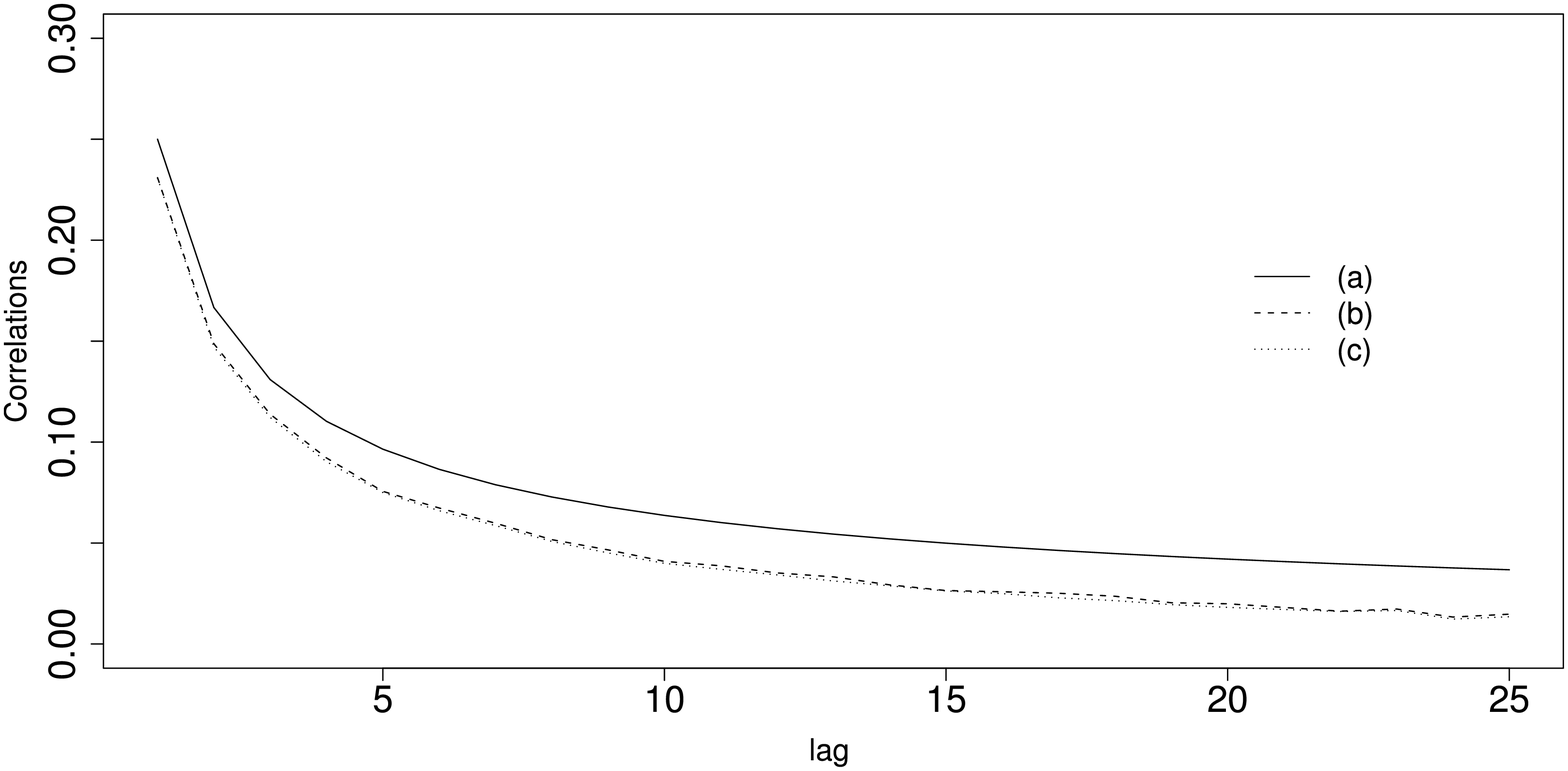}
&\includegraphics*[width=7cm,height=5cm]{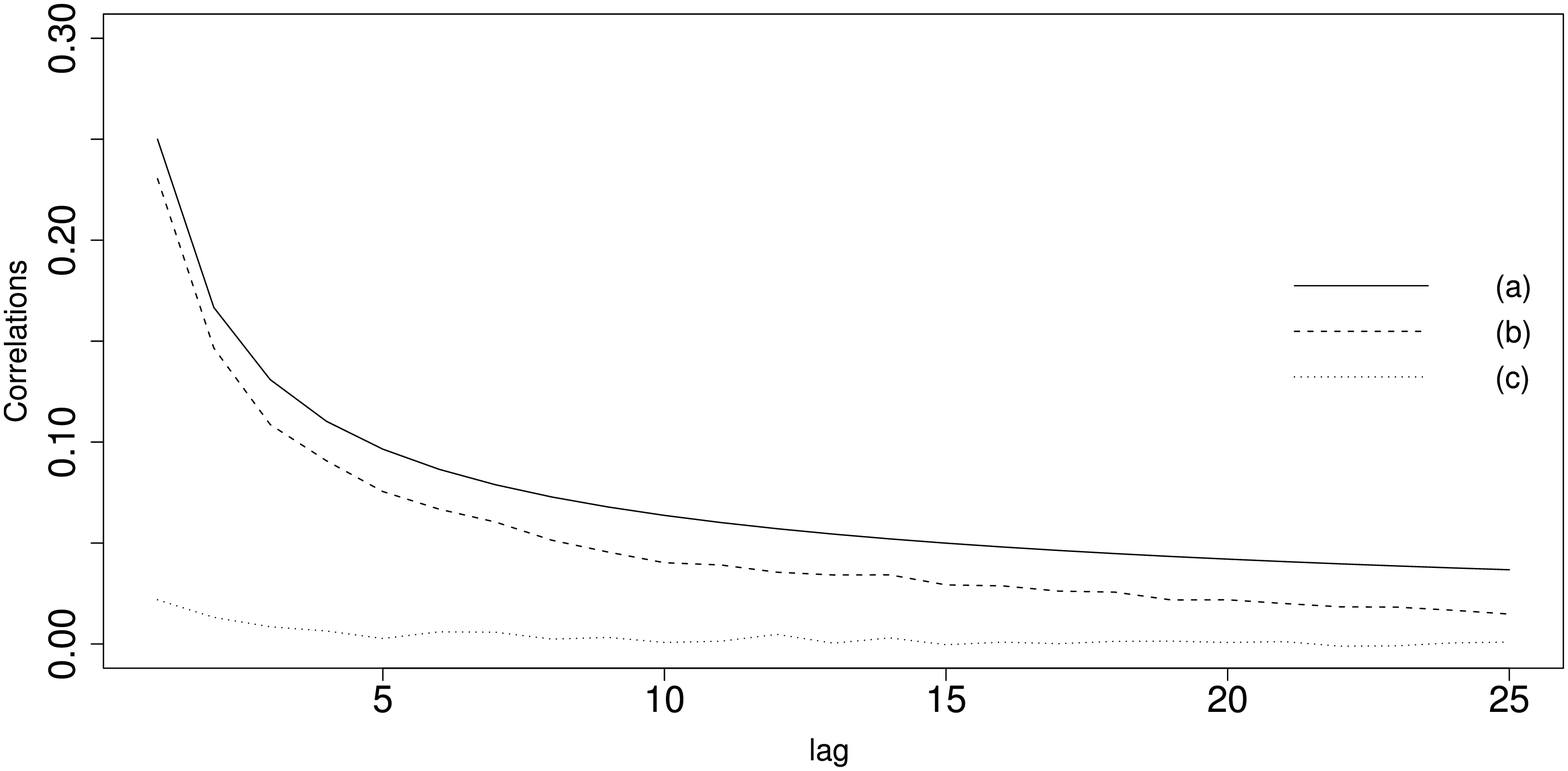}
\end{tabular}
\caption{\footnotesize{Sample correlations of the ARFIMA$(0,d,0)$
model with $d =0.2$, $n$=500 without outliers (left) and with outliers
with $p=10\%$
and $\omega =10$ (right).
(a) is the population correlation, and (b) and (c) are
the robust and the classical sample autocorrelations, respectively.}}
\label{fig:autocorr_simul}
\end{figure}

 In the absence of atypical observations (left part of Figure \ref{fig:autocorr_simul}), both
sample functions display a similar behavior. However, when the data
contains outliers (right part of Figure \ref{fig:autocorr_simul})
the classical sample autocorrelation is clearly distorted.

\subsection{Non-Gaussian observations}

We now examine the behavior of the autocovariance
estimator when it is applied to non Gaussian observations. 
To do so, we generated observations
$(X_t)_{1\leq t\leq n}$
as follows,
$$
X_t=\phi_1 X_{t-1}+Z_t\; ,
$$
where $\phi_1=0.9$, $\varepsilon=0.4$, $Z_t=W_t+\varepsilon Y_t^2$,  where $W_t$ and $Y_t$ are
independent random variables such that $(W_t)$ and $(Y_t)$ are i.i.d 
standard Gaussian random variables. An example of a realization of
$(Z_t)_{1\leq t\leq n}$ is given in the histogram of Figure
\ref{fig:non_gauss} with $n=500$. As we can see from this figure, the
presence of $\varepsilon$ in the definition of $Z_t$ produces an
asymmetry in the data. In the right part of this figure, we displayed 
the average of the robust autocovariance
$\widehat{\gamma}_Q(h,\chunk{X}{1}{n},\Phi)$ in (\ref{def:gamma_Q})
and the classical
autocovariance $\widehat{\gamma}(h)$ defined 
in Remark \ref{rem:autocov_short}, for $h=1,\dots,40$ and 1000 replications. From this
figure, we can see that the robust autocovariance estimator does not
seem to be affected by the skewness of the data.

\begin{figure}[!ht]
\begin{tabular}{cc}
\includegraphics*[width=7cm,height=8cm]{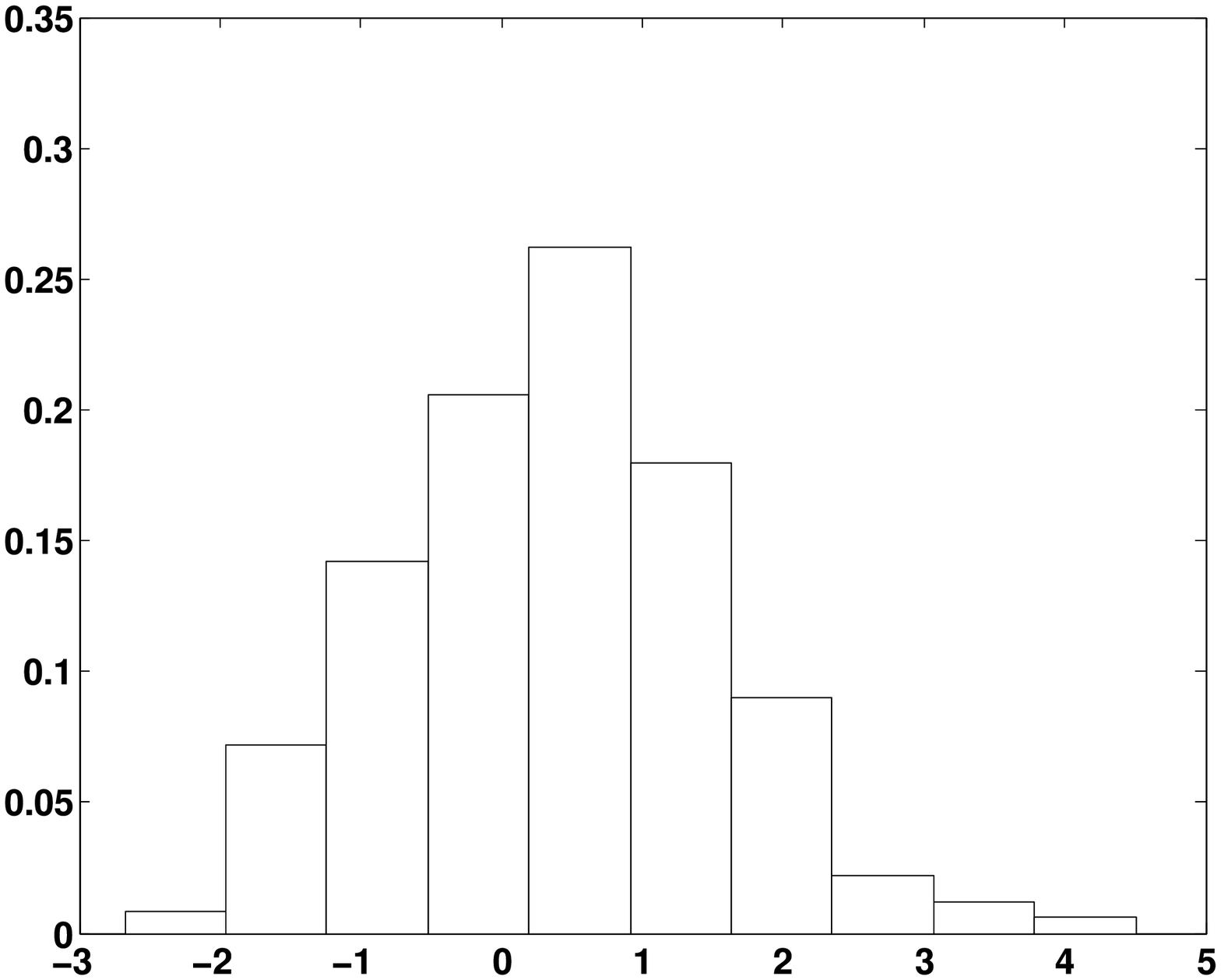}
&\includegraphics*[width=8cm, height=8cm]{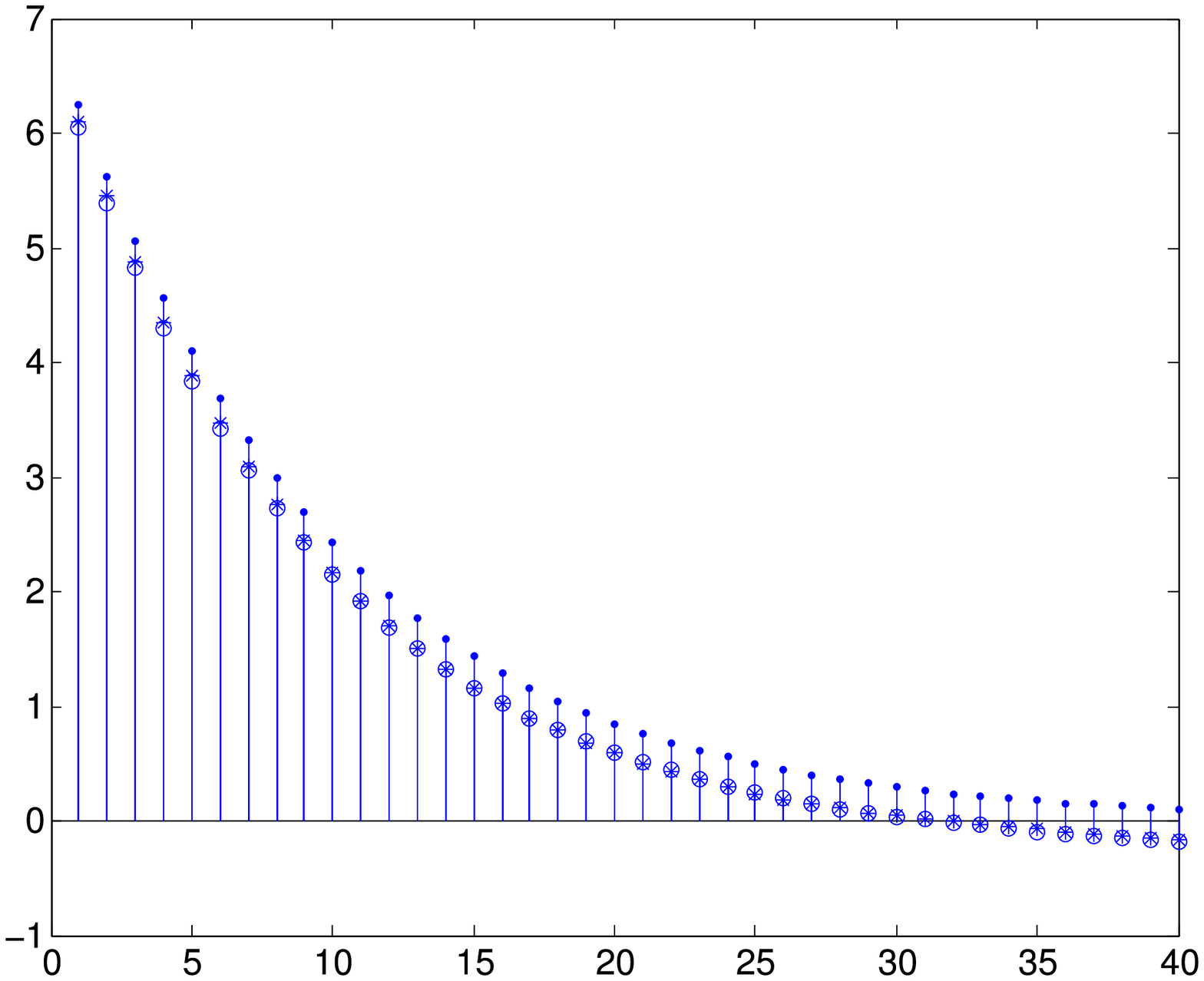}
\end{tabular}
\caption{\footnotesize{Left: Histogram of one realization of
    $(Z_t)_{1\leq t\leq 500}$. Right: Theoretical
    autocovariance('.'), robust autocovariance ('$\circ$'), classical
    autocovariance ('$\star$') for $h=1,\dots,40$.}}
\label{fig:non_gauss}
\end{figure}

\section{An application}
\label{sec:appli}
The Nile data is used here to illustrate some of the robust methodologies
discussed previously.
The Nile River data set is a well-known and interesting time series, which has
been extensively analyzed. This data is discussed in detail in the book by \cite{beran:1994}.
It is first introduced in Section~1.4 on p.~20, and is completely
tabulated on pp.~237--239. \cite{beran:1994} took this data from an earlier book by \cite[pp. 366--404]{toussoun:1925}.
The data consists of yearly minimal water levels of the Nile river measured at the
Roda gauge, near Cairo,  for the years 622--1284~AD (663 observations);
The units for the data as presented by \cite{beran:1994} are
centimeters (presumably above some fixed reference point).
The empirical mean and the standard
deviation of the data are equal to 1148 and 89.05, respectively.

The question has been raised as to whether the Nile time series
contains  outliers; see for example \cite{beran:1992},
\cite{robinson:1995}, \cite{chareka:matarise:turner:2006} and \cite{fajardo:reisen:cribari:2009}.
The test procedure developed by \cite{chareka:matarise:turner:2006},
suggests the presence of outliers at 646 AD ($p$-value 0.0308) and at
809 ($p$-value 0.0007). Another possible outlier is at 878 AD.
A plot of the time series where the observations which have been judged to be
outliers are marked, is shown in the left part of Figure
\ref{fig:appli}, and the right part of this figure displays the
histogram of the data. Although the theory developed in this paper is
related to Gaussian processes, we believe that the small asymmetry
of the data does not compromise the use of this series as an
illustration of our robust methodology. A way to avoid
this asymmetry is to consider the logarithm of the data. However,
this does not make a significant difference in the estimates.

\begin{figure}[!h]
\begin{tabular}{cc}
\includegraphics*[width=7.5cm]{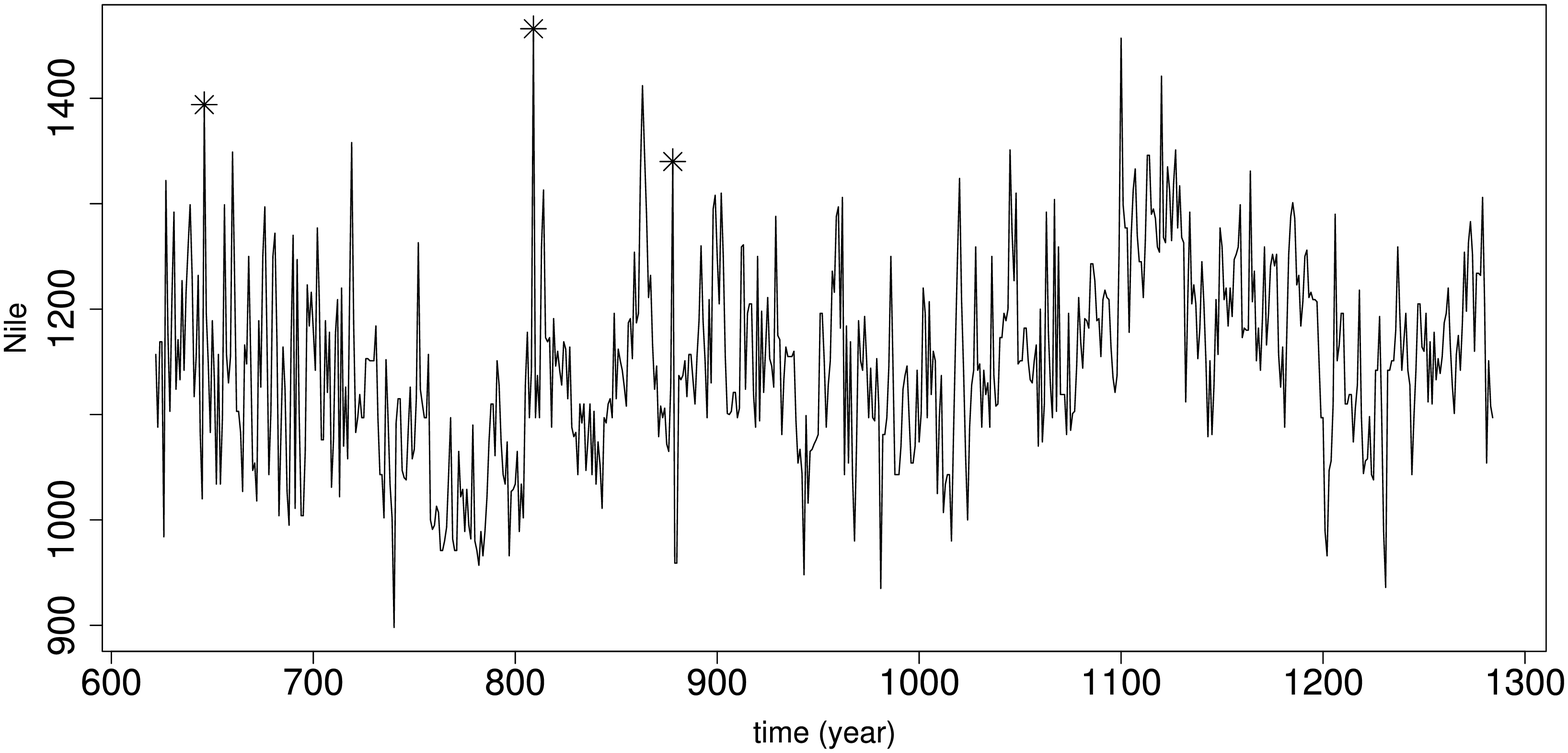}
&\includegraphics*[width=7.5cm]{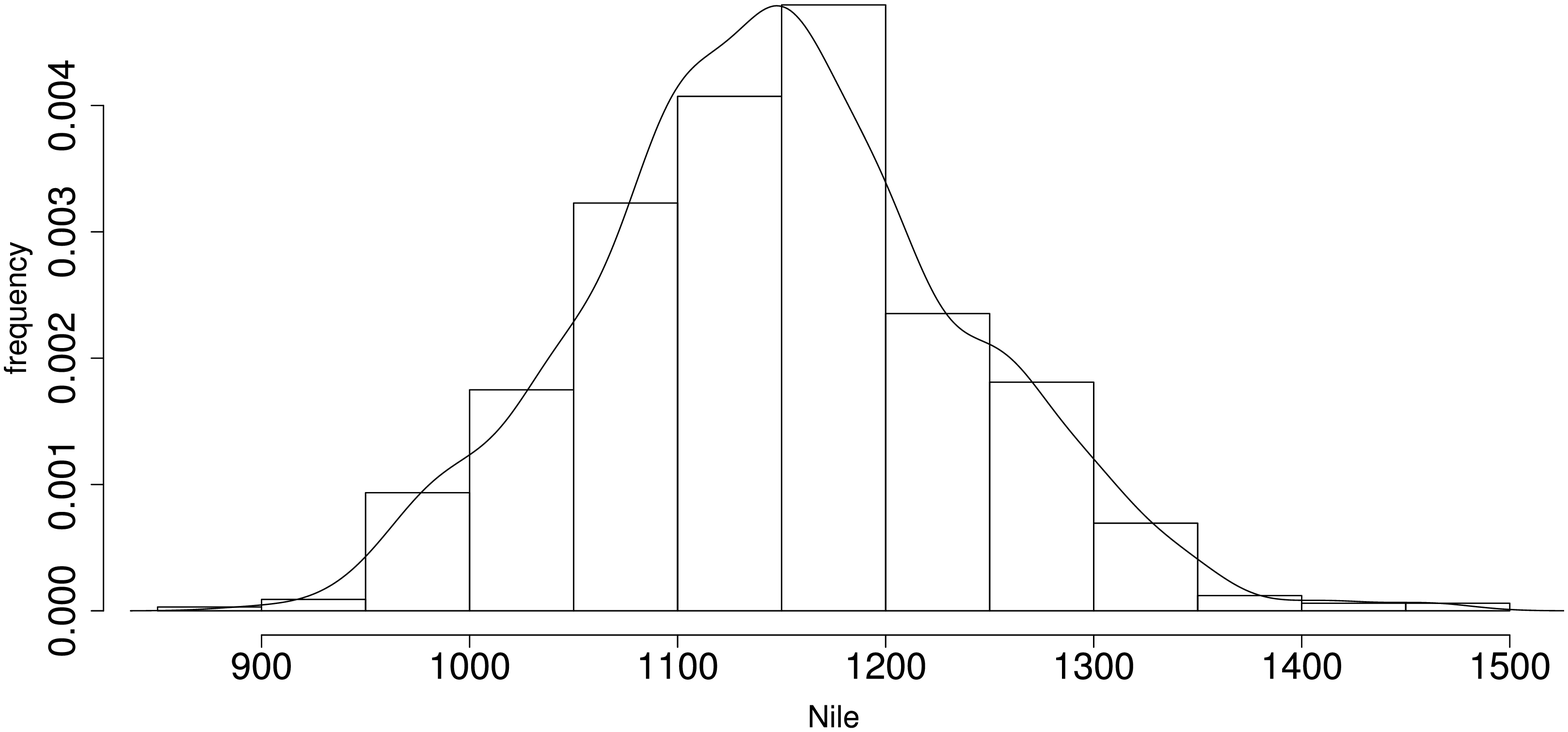}
\end{tabular}
\caption{Left: The Nile River data plot. Right: Histogram of the Nile River data.}
\label{fig:appli}
\end{figure}

The left part of Figure \ref{fig:appli_2} displays plots of the
classical and robust sample autocorrelation functions of the original data.
The autocorrelation values from the former are smaller
than those of the latter one. However, the difference between the  autocorrelations
may be not large enough to suggest the presence of outliers.
Thus, to better understand the influence of
outliers on the sample autocorrelation functions in practical
situations a new dataset with
artificial outliers was generated. 
We replaced the presumed outliers detected
by \cite{chareka:matarise:turner:2006} by the mean plus 5 or 10 standard deviations.
The sample autocorrelations (robust and classical ones) were again
calculated, see the right part of Figure \ref{fig:appli_2}.
As expected, the values of the robust autocorrelations remained
stable. However, the classical autocorrelations were significantly
affected by the increase of the size of the observation. This is
in accordance with the results presented in the simulation section.



\begin{figure}[!h]
 \begin{tabular}{cc}
\hspace{-3mm}
 \includegraphics*[width=7.75cm,height=6cm]{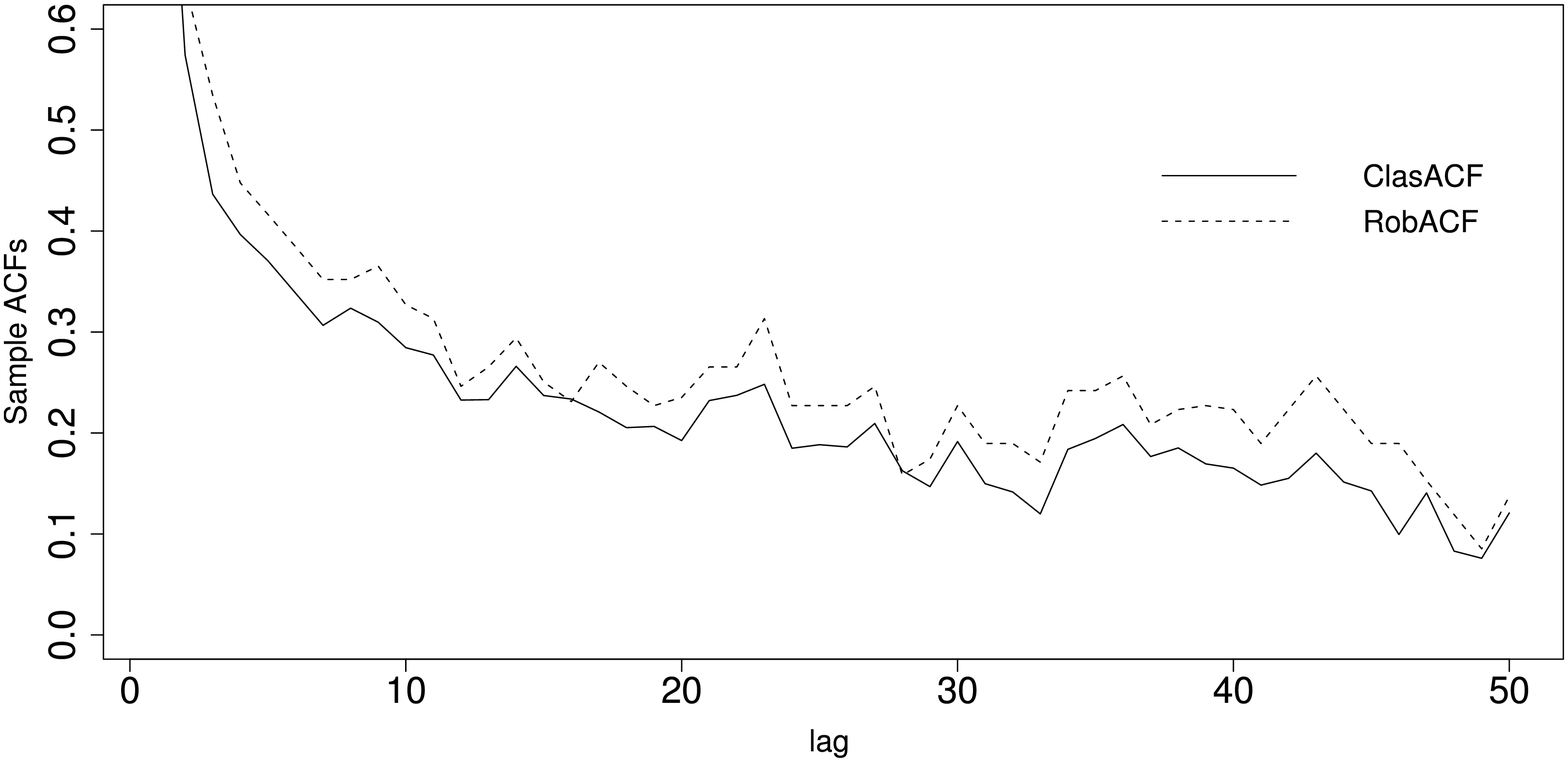}
 &\includegraphics*[width=7.75cm,height=6cm]{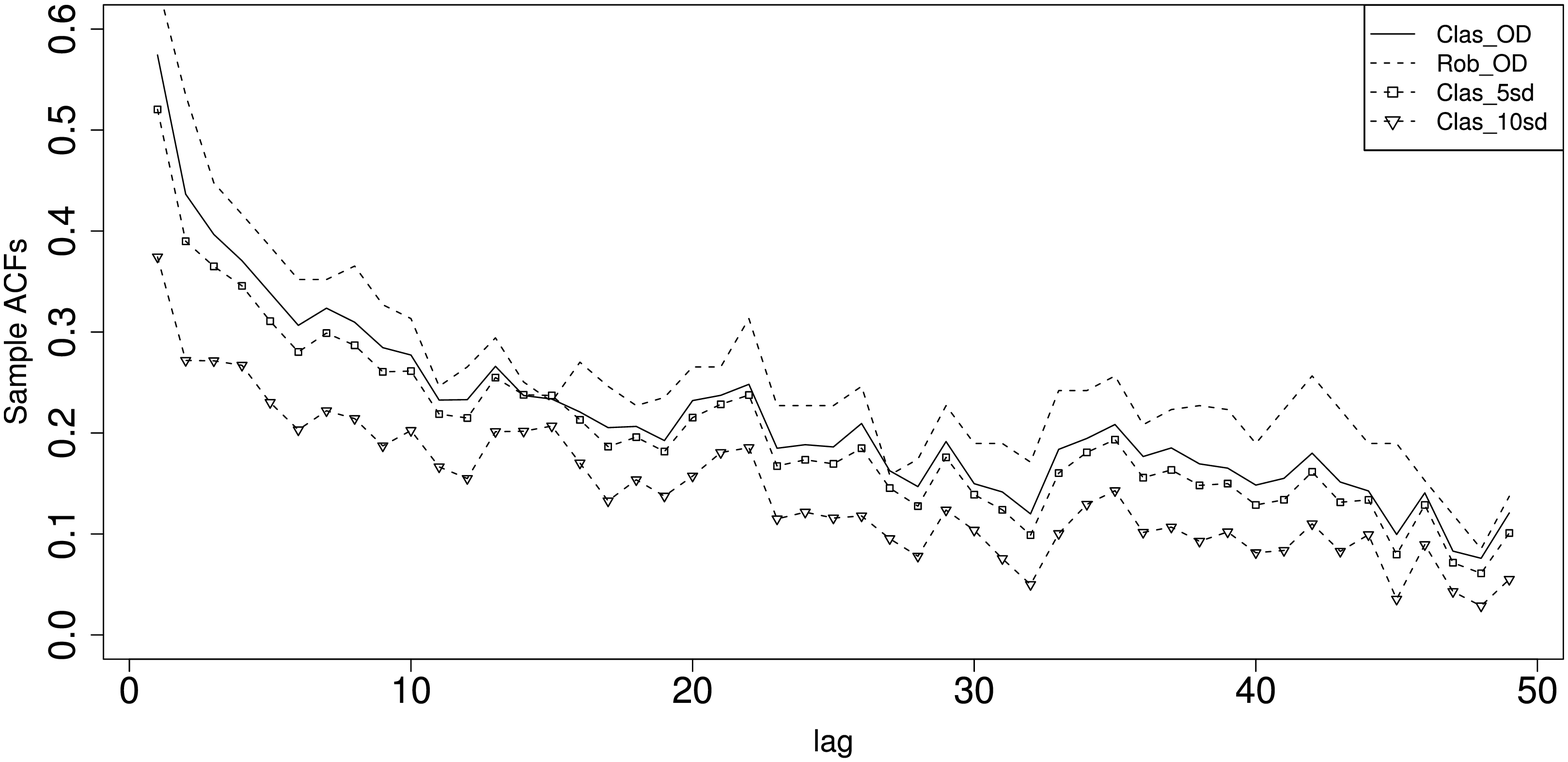}
 \end{tabular}
\caption{Left: Classical (plain line) and robust (dotted line) sample autocorrelation functions
  of the Nile River data. Right: Classical (plain line) and robust (dotted line) sample autocorrelation functions
  of the Nile River data and classical  sample autocorrelation
  functions of the Nile River data with artificial outliers at the
  times detected by \cite{chareka:matarise:turner:2006} (original
  data plus 5 standard deviations with `` $\Box$ "  and original
  data plus 10 standard deviations with `` $\triangledown$ ").}
\label{fig:appli_2}
\end{figure}

\section{Asymptotic behavior of $U$-processes}\label{s:U-process}

Consider the $U$-process $\{U^G_n(r), r\in I\}$ satisfying
\begin{equation}\label{def:U_n}
U^G_n(r)=\frac{1}{n(n-1)}\sum_{1\leq i\neq j\leq
  n}\1_{\{G(X_i,X_j)\leq r\}}\; , r\in I
\end{equation}
based on the class of kernels
\begin{equation}
\label{eq:definition-k}
k_G(x,y,r)= \1_{\{ G(x,y) \leq r\}} \eqsp.
\end{equation}
where $I$ is an interval included in $\rset$, $G$ is a symmetric function \textit{i.e.}
$G(x,y)=G(y,x)$ for all $x,y$ in $\rset$, and the process $(X_i)_{i\geq 1}$
satisfies Assumption (A\ref{assum:long-range}) with $\gamma(0)=1$.

The asymptotic properties of these $U$-processes have been studied in \cite{boistard:levy:2009}. 
They are based on the computation of the Hermite rank of the
class of functions $\{\1_{\{G(\cdot,\cdot) \leq r\}}-U^G(r), r\in I\}$ where
\begin{equation}\label{def:U_hoeff}
 U^G(r)=\int_{\rset^2} \1_{\{ G(x,y) \leq r \}} \phi(x) \phi(y) \rmd x \rmd y\;,\; \text{for all} \; r\in I\; .
\end{equation}
The Hermite rank of the class of functions $\{\1_{\{G(\cdot,\cdot)\leq
  r\}} -U(r), r\in I\}$ is obtained by 
expanding the function  $\1_{\{G(\cdot,\cdot)\leq r\}}$ in the basis of Hermite
polynomials with leading coefficient equal to 1:
\begin{equation}\label{eq:hermite_decomp}
\1_{\{ G(x,y) \leq r\}} =\sum_{p,q\geq 0}\frac{\alpha_{p,q}(r)}{p!q!} H_p(x) H_q(y)\;
, \textrm{ for all } x,y \textrm{ in } \rset\; ,
\end{equation}
where $\alpha_{p,q}(r)=\PE\left[\1_{\{ G(X,Y) \leq r\}}H_p(X) H_q(Y)\right]$, $X$ and $Y$ being
independent standard Gaussian random variables. The first few Hermite
polynomials are $H_0(x)=1$, $H_1(x)=x$, $H_2(x)=x^2-1$, $H_3(x)=x^3-3x$. Note that
$\alpha_{0,0}(r)$ is equal to $U^G(r)$ for all $r$, where $U^G(r)$ is defined in
(\ref{def:U_hoeff}).
The previous expansion can also be rewritten as
\begin{equation}\label{eq:hermite_rank}
\1_{\{ G(x,y) \leq r\}}-U^G(r)=\sum_{\stackrel{p,q\geq 0}{p+q\geq m}}\frac{\alpha_{p,q}(r)}{p!q!}
H_p(x) H_q(y)\; ,
\end{equation}
where
$m=m(r)$ is called the Hermite rank of the
function $\1_{\{G(\cdot,\cdot) \leq r\}}-U^G(r)$ when $r$ is fixed.

We state the results for family of kernels having Hermite rank equal
to $m=2$ 
(this is all we need here) and refer to \cite{boistard:levy:2009}
for other cases.

\begin{prop}\label{theo:U_n_D>1/2}
Let $I$ be a compact interval of $\rset$, let $k_G(\cdot,\cdot,r)$ be
defined in (\ref{eq:definition-k}), and let
\begin{eqnarray}\label{def:h_1}
  k_{G,1}(x,r)= \PE \left[k_G(x,Y,r) \right] \eqsp,\; x\in\rset\; ,\;
  r\in I\; ,
\end{eqnarray}
where $Y$ is a standard Gaussian variable.
Suppose that the Hermite rank of the class of functions $\{k_G(\cdot,\cdot,r)-U^G(r)\; ,
r\in I\}$ is $m=2$ and that Assumption (A\ref{assum:long-range}) is satisfied
with  $\gamma(0)=1$ and $1/2<D<1$. Assume that $k_G$ 
satisfies the following three conditions:
\begin{enumerate}
\item[(i)] There exists a positive constant $C$ such that for all $s$, $t$ in $I$, $u$, $v$ in $\rset$,
\begin{equation}\label{cond:contr_G}
\PE\left[|k_G(X+u,Y+v,s)-k_G(X+u,Y+v,t)|\right]\leq\; C|t-s|\; ,
\end{equation}
where $(X,Y)$ is a standard Gaussian random vector.
\item[(ii)]  There exists a positive constant $C$ such that for all 
$\ell\geq 1$ and $s,t$ in $I$, $u$, $v$ in $\rset$,
\begin{equation}\label{eq:G_lip_2}
\PE\left[|k_G(X_1+u,X_{1+\ell}+v,t)-k_G(X_1,X_{1+\ell},t)|\right]\leq C |u-v|\; .
\end{equation}
\begin{equation}\label{eq:G_lip}
\PE\left[|k_G(X_1,X_{1+\ell},s)-k_G(X_1,X_{1+\ell},t)|\right]\leq C |t-s|\; ,
\end{equation}
\item[(iii)] There exists a positive constant $C$ such that 
for all  $t$ in $I$, and $x$, $u$, $v$ in $\rset$,
\begin{align}
\label{cond:h_1}
&|k_{G,1}(x+u,t)-k_{G,1}(x+v,t)|\leq C|u-v| \eqsp, \\
\label{assum:h_1_Lipschitz}
&|k_{G,1}(x,s)-k_{G,1}(x,t)|\leq C|t-s| \eqsp.
\end{align}
\end{enumerate}
Then the $U$-process $\{\sqrt{n}(U^G_n(r)-U^G(r)), r\in I\}$ defined in
\eqref{def:U_n} and \eqref{def:U_hoeff} converges weakly in the space
of cadlag functions  on $I$, $\mathcal{D}(I)$,
 equipped with the
topology of uniform convergence to the zero mean Gaussian process $\{W_G(r),r\in I\}$ with
covariance structure given by
\begin{multline}\label{eq:cov_struct}
\PE[W_G(s)W_G(t)]=4\;\Cov(k_{G,1}(X_1,s),k_{G,1}(X_1,t))\\
+4\sum_{\ell\geq1}\ \Cov(k_{G,1}(X_1,s),k_{G,1}(X_{\ell+1},t))+\Cov(k_{G,1}(X_1,t),k_{G,1}(X_{\ell+1},s))\; .
\end{multline}
Moreover, for a fixed $r$ in $I$, as $n$ tends to infinity,
\begin{equation}\label{eq:dev_UG}
\sqrt{n}(U^G_n(r)-U^G(r))
=\frac{2}{\sqrt{n}}\sum_{i=1}^n
\left[k_{G,1}(X_i,r)-U^G(r)\right]+o_P(1)\; .
\end{equation} 
\end{prop}

We now consider the case where $D < 1/2$. In this case, the normalization depends, as expected, on $D$ and the slowly
varying function $L$ and the limiting distribution is no longer a Gaussian process.
Let $(Z_{1,D}(t))_{t \in \rset_+}$ denote the standard fractional Brownian motion (fBm)
and $(Z_{2,D}(t))_{t \in \rset_+}$ the Rosenblatt process. They are
defined through multiple Wiener-It\^o  integrals and given by
\begin{equation}\label{eq:fBm}
Z_{1,D}(t)=\int_{\rset}\left[\int_0^t (u-x)_+^{-(D+1)/2} \rmd u\right] \rmd B(x),\quad 0<D<1\; ,
\end{equation}
and
\begin{equation}\label{eq:rosenblatt}
Z_{2,D}(t)=\int'_{\rset^2}\left[\int_0^t (u-x)_+^{-(D+1)/2}
(u-y)_+^{-(D+1)/2} du\right] \rmd B(x) \rmd B(y),\; 0<D<1/2\; ,
\end{equation}
where $B$ is the standard Brownian motion, see \cite{fox:taqqu:1987}.
The symbol $\int'$ means that the domain of integration excludes the
diagonal.
Introduce also the Beta function
\begin{equation}\label{eq:betaf}
\B(\alpha,\beta)=\int_0^{\infty} y^{\alpha-1}(1+y)^{-\alpha-\beta}\rmd
y = \frac{\Gamma(\alpha)\Gamma(\beta)}{\Gamma(\alpha+\beta)}\; , 
\quad\alpha >0,\; \beta>0\; .
\end{equation}

\begin{prop}\label{theo:D<1/2}
Let $I$ be a compact interval of $\rset$
Suppose that the Hermite rank of the class of functions $\{k_G(\cdot,\cdot,r)-U^G(r)\; ,
r\in I\}$ is $m=2$ and that Assumption (A\ref{assum:long-range}) is satisfied
with $\gamma(0)=1$ and $D<1/2$. Assume the following:
\begin{enumerate}
\item [(i)] There exists a positive constant C such that,
for all $\ell\geq 1$ and for all $s,t$ in $I$,
\begin{equation}\label{eq:lip_h}
\PE[|k_G(X_1,X_{1+\ell},s)-k_G(X_1,X_{1+\ell},t)|]\leq C |t-s|\; .
\end{equation}
\item [(ii)] $U^G$ is a Lipschitz function
\item [(iii)] The function $\widetilde{\Lambda}$ defined, for all $s$ in $I$, by
\begin{equation}\label{eq:lambda_tilde}
\widetilde{\Lambda}(s)=\PE[k_G(X,Y,s)(|X|+|XY|+|X^2-1|)]\; ,
\end{equation}
where $X$ and $Y$ are independent standard Gaussian random variables,
is also a Lipschitz function.
\end{enumerate}

Then, the $U$-process $\{U_n^G(r)-U^G(r), r\in I\}$
defined in (\ref{def:U_n}) and (\ref{def:U_hoeff}) has
the following asymptotic properties:
$$
\left\{n^{D} L^{-1}(n) \left(U_n^G(r)-U^G(r)\right); r\in I\right\}
$$
converges weakly in the space of cadlag functions $\mathcal{D}(I)$, equipped with the
topology of uniform convergence, to
$$
\{\beta(D)^{-1}\left[\alpha_{1,1}(r)Z_{1,D}(1)^2+\alpha_{2,0}(r)Z_{2,D}(1)\right]; r\in I\}\; ,
$$
where the fractional Brownian motion $Z_{1,D}(\cdot)$ and the
Rosenblatt process $Z_{2,D}(\cdot)$ are defined in (\ref{eq:fBm}) and
(\ref{eq:rosenblatt}) respectively and where $\beta(D)=\emph{\B}((1-D)/2,D)$, $\emph{\B}$ denoting the Beta
function, defined in (\ref{eq:betaf}).
\end{prop}

Propositions \ref{theo:U_n_D>1/2} and \ref{theo:D<1/2} will be applied
to the $U$-process $U_n(r)$ in (\ref{eq:definition-U-process}) with
\begin{equation}\label{e:U}
U(r)=\int_{\mathbb{R}^2}\1_{\{\|x-y|\leq r\}} \rmd F(x)\rmd
F(y)=T_1(F)[r]\; ,
\end{equation}
with $T_1$ given in (\ref{eq:definition-T1}). 
By Lemma \ref{lem:hermite_rank_T1(F_n)}, 
the Hermite rank of the class of functions $\{\1_{|x-y| \leq r} -U(r),
x,y \in \rset, r \in I\}$
is equal to 2 where $I=[r_0-\eta,r_0+\eta]$ for some positive
$\eta$ defined in Lemma \ref{lem:hermite_rank_T1(F_n)} and where 
$r_0=1/c(\Phi)$ (see (\ref{eq:c(Phi)})) is such that
\begin{equation}\label{eq:r_0}
T_1(\Phi)[r_0]=T_1(\Phi)[1/c(\Phi)]=1/4\; .
\end{equation}

\section{Proofs}\label{sec:proofs}

\begin{proof}[Proof of Lemma \ref{prop:asym_expansion}]
Denote by $F$ the c.d.f. $\Phi_{\mu,\sigma}$ of $X_1$.
Since $a_n(F_n-F)$ converges in distribution
in the space of cadlag functions equipped with the topology of uniform convergence,
the asymptotic expansion \eqref{expansion:short-range} can be deduced from
the functional Delta method stated \textit{e.g.}
in Theorem 20.8 of \cite{vdv:2000}. To show this, we
have to prove
that $T_0=T_1 \circ T_2$ is Hadamard differentiable, where $T_1$ and $T_2$ are
defined in \eqref{eq:definition-T1} and \eqref{eq:definition-T2} respectively and that the corresponding
Hadamard differential is defined and continuous on the whole space of cadlag
functions. For a definition of Hadamard differentiability, we refer to
\cite[Chapter 20]{vdv:2000}.

We prove first that the Hadamard differentiability of the functional 
$T_1$ defined in \eqref{eq:definition-T1}. Let $(g_t)$
be a sequence of cadlag functions with bounded
variations such that $\supnorm{g_t-g}\to 0$,  as $t\to 0$, where $g$ is a cadlag function.
For any non negative $r$, we consider
\begin{multline*}
t^{-1}\left\{T_1(F+tg_t)[r]-T_1(F)[r]\right\}\\
=2\int_{\rset}\int_{\rset}\1_{\{\vert x-y\vert\leq r\}}\rmd F(x)\rmd g_t(y)
+t\int_{\rset} \int_{\rset} \1_{\{\vert x-y\vert\leq r\}}\rmd g_t(x)\rmd g_t(y)\; .
\end{multline*}
Since
\begin{multline*}
\left|\int_{\rset}\int_{\rset}\1_{\{|x-y|\leq r\}}\rmd F(x)\rmd g_t(y)
-\int_{\rset}\int_{\rset}\1_{\{|x-y|\leq r\}}\rmd F(x)\rmd g(y)\right|\\
=\left|\int_{\rset} \left(g_t(x+r)-g(x+r) \right)\rmd F(x)
-\int_{\rset} \left(g_t(x-r)-g(x-r) \right)\rmd F(x)\right|
\leq 2 \supnorm{g_t-g}\to 0\; ,
\end{multline*}
as $t$ tends to zero, the Hadamard differential of $T_1$ at $g$ is given by:
$$
(DT_1(F).g)(r)
= 2\int_{\rset}\int_{\rset}\1_{\{|x-y|\leq r\}}\ \rmd F(x) \rmd g(y)
=2\int_{\rset} \{g(x+r)-g(x-r)\} \rmd F(x)\; .
$$
By Lemma 21.3 in \cite{vdv:2000}, $T_2$ is Hadamard differentiable. Finally, using
the Chain rule (Theorem 20.9 in \cite{vdv:2000}), we obtain the Hadamard differentiability
of $T_0$ with the following Hadamard differential:
\begin{equation}\label{e:DT}
DT_0(F).g=-\frac{(DT_1(F).g)(T_0(F))}{(T_1(F))'[T_0(F)]}
=-\frac{2\int_{\rset} \{g(x+T_0(F))-g(x-T_0(F))\} \rmd F(x)}{(T_1(F))'[T_0(F)]}\; .
\end{equation}
In view of the last expression, $DT_0(F)$ is a continuous
function of $g$ and is defined on the whole space of cadlag functions.
Thus,  by Theorem 20.8 of \cite{vdv:2000}, we obtain:
\begin{equation}\label{e:QFnF}
a_n(\QRC[][\Phi]{n}{\chunk{X}{1}{n}}-Q(F))= c(\Phi)\; DT_0(F).\{a_n(F_n-F)\}+o_P(1)\; ,
\end{equation}
where $c(\Phi)$ is the constant defined in (\ref{eq:c(Phi)}).
By (\ref{eq:definition-T1}), 
$T_1(F)[r]=\int_{\rset} [F(x+r)-F(x-r)]\rmd F(x)$ and since
$F(\cdot)=\Phi_{\mu,\sigma}(\cdot)=\Phi((\cdot-\mu)/\sigma)$, we get
$$
(T_1(F))'[r]=\frac{2}{\sigma}\int_{\rset}\Phi\left(y+\frac{r}{\sigma}\right)
\Phi(y)\rmd y\; .
$$
Since $\sigma=Q(\Phi_{\mu,\sigma})=c(\Phi)T_0(F)$ by (\ref{eq:def:Q}),
we get
\begin{equation}\label{e:T1F}
(T_1(F))'[T_0(F)]=2\sigma^{-1}\int \phi(y)\phi(y+1/c(\Phi))\rmd y\; .
\end{equation}
Applying (\ref{e:DT}) with $T_0(F)=\sigma/c(\Phi)$, using (\ref{e:T1F}), and setting $g=a_n(F_n-F)$, we get
\begin{equation}\label{e:DAB}
DT_0(F).\{a_n(F_n-F)\}=A_n-B_n\; ,
\end{equation}
where 
\begin{equation}\label{e:An}
A_n=a_n\left((T_1(F))'[T_0(F)]\right)^{-1}
\int_{\rset}\left\{F\left(x+\frac{\sigma}{c(\Phi)}\right)
-F\left(x-\frac{\sigma}{c(\Phi)}\right)\right\}\rmd F(x)
\end{equation}
and $B_n$ has the same expression with $F$ replaced by $F_n$.
The integral in $A_n$ equals
\begin{equation}\label{e:e2}
\int_{\rset}\int_{\rset} \1_{\{|y-x|\leq\sigma/c(\Phi)\}} \rmd F(x)
\rmd F(y)=T_1(F)[T_0(F)]=1/4\; ,
\end{equation}
by definition (see (\ref{eq:def:Q})). The corresponding integral in
$B_n$ equals
\begin{multline*}
\frac{1}{n}\sum_{i=1}^n 
\left\{F\left(X_i+\frac{\sigma}{c(\Phi)}\right)
-F\left(X_i-\frac{\sigma}{c(\Phi)}\right)\right\}\\
=\frac{1}{n}\sum_{i=1}^n 
\left\{\Phi\left(\frac{X_i-\mu}{\sigma}+\frac{1}{c(\Phi)}\right)
-\Phi\left(\frac{X_i-\mu}{\sigma}-\frac{1}{c(\Phi)}\right)\right\}\; .
\end{multline*}
The result follows from (\ref{e:QFnF}), (\ref{e:DT}),
(\ref{e:DAB}), (\ref{e:T1F}) and the above expressions for $A_n$ and $B_n$.
\end{proof}

\begin{proof}[Proof of Theorem \ref{theo:short-range}]
Assumption (A\ref{assum:short-range}) and the Theorem of \cite{csorgo:mielniczuk:1996}
implies that
$\sqrt{n}(F_n-\Phi_{0,\sigma})$ converges in distribution to a Gaussian process
in the space of cadlag functions equipped with the topology of uniform convergence.
Thus, the asymptotic expansion of $a_n(\QRC[][\Phi]{n}{\chunk{X}{1}{n}}-\sigma)$ obtained in
\eqref{expansion:short-range} is valid with $a_n=\sqrt{n}$.
We thus have to prove a CLT for $n^{-1/2}\sum_{i=1}^n \IF(X_i/\sigma,Q,\Phi)$.
Using Lemma \ref{lem:hermite_Q} below, we note
that the Hermite rank of $\IF(\cdot,Q,\Phi)$ is equal to 2 and the
conclusion follows by applying \cite[Theorem 1]{breuer:major:1983}.
\end{proof}

\begin{proof}[Proof of Proposition \ref{p:sigma}]
Note that $\widehat{\sigma}^2_{n,X}=\gamma(0) \widehat{\sigma}^2_{n,Y}$, where $(Y_i)_{i\geq 1}$ satisfies
(A\ref{assum:short-range}) with $\gamma(0)=1$.
Observe that $\widehat{\sigma}^2_{n,Y}-1$ is a $U$-statistic with kernel $k(x,y)=
(x-y)^2/2- 1$.
The Hoeffding decomposition of this kernel is given by $k(x,y)=
(x^2-1)/2 + (y^2-1)/2 - x y$.
From this, we obtain the corresponding Hoeffding decomposition of $\widehat{\sigma}^2_{n,Y}-1$ as
\begin{equation}\label{eq:hoeff_dec_sigma_n}
 \widehat{\sigma}^2_{n,Y}-1= \frac{1}{n} \sum_{i=1}^{n} H_2(Y_i)-
 \frac{1}{n(n-1)} \sum_{1 \leq i \ne j \leq n} Y_i Y_j \eqsp.
\end{equation}
Under Assumption (A\ref{assum:short-range}), the first term of this
decomposition is the leading one. Then, using
\cite[Theorem 1]{breuer:major:1983},
we get that $\sqrt{n} (\widehat{\sigma}^2_{n,X} - \sigma^2)$ converges to a zero-mean Gaussian
random variable having a variance equal to
$2(\gamma(0)^2+2\sum_{k\geq 1}\gamma(k)^2)$. Using the Delta method to
go from $\widehat{\sigma}^2_{n,X}$ to $\widehat{\sigma}_{n,X}$, setting
$f(x)=\sqrt{x}$, so that
$f'(\sigma^2)=1/(2\sqrt{\sigma^2})=1/(2\sigma)$, 
we get that 
the asymptotic variance of $\sqrt{n} \left(\widehat{\sigma}_{n,X} - \sigma\right)$
is thus equal to (\ref{e:lvar1}).

By Lemma \ref{lem:hermite_Q}, the Hermite rank of $\IF(.,Q,\Phi)$ is
equal to 2, hence using  \cite[Lemma 1]{Arcones:1994},
$\widetilde{\sigma}^2$ defined in (\ref{def:sigma_tilde}) satisfies
$$
\widetilde{\sigma}^2\leq\gamma(0)^{-1}\PE[\IF(X_1/\sigma,Q,\Phi)^2] \{\gamma(0)^2+2 \sum_{k\geq
  1}\gamma(k)^2\}\; .
$$
Finally, in this case, using that $\PE[\IF(X_1/\sigma,Q,\Phi)^2]\approx
0.6077$ \cite[p. 1278]{Rousseeuw:Croux:1993},
the relative asymptotic efficiency $\widetilde{\sigma}^2_{cl}/\widetilde{\sigma}^2$ of
$\QRC[][\Phi]{n}{\chunk{X}{1}{n}}$ compared to
$\widehat{\sigma}_{n,X}$ is larger than $82.27\%$ since
$$
\frac{(2\gamma(0))^{-1}(\gamma(0)^2+2\sum_{k\geq
    1}\gamma(k)^2)}{\gamma(0)^{-1}\PE[\IF(X_1/\sigma,Q,\Phi)^2]
  \{\gamma(0)^2+2 \sum_{k\geq 1}\gamma(k)^2\}}\approx
0.5/0.6077\approx 82.27\%\;.
$$
\end{proof}

\begin{proof}[Proof of Theorem \ref{theo:gamma_short}]

Let $\Phi_{\sigma,+}$ and $\Phi_{\sigma,-}$ denote the c.d.f of
$(X_i+X_{i+h})_{i\geq 1}$ and $(X_i-X_{i+h})_{i\geq 1}$, respectively.
Let also denote by $F_{+,n-h}$ and $F_{-,n-h}$ 
the empirical c.d.f of $(X_i+X_{i+h})_{1\leq i\leq n-h}$
and $(X_i-X_{i+h})_{1\leq i\leq n-h}$, respectively. 
Since
$(X_i)_{i\geq 1}$ satisfy Assumption (A\ref{assum:short-range}), it is
the same for $(X_i+X_{i+h})_{i\geq 1}$ and $(X_i-X_{i+h})_{i\geq 1}$
with scales equal to $Q(\Phi_{\sigma,+})$ and $Q(\Phi_{\sigma,-})$,
respectively.
Thus, using the Theorem of \cite{csorgo:mielniczuk:1996}, we
obtain that $\sqrt{n-h}(F_{+,n-h}-\Phi_{\sigma,+})$ converges in distribution to a Gaussian process
in the space of cadlag functions equipped with the topology of uniform
convergence and that the same holds for $\sqrt{n-h}(F_{-,n-h}-\Phi_{\sigma,-})$.
As a consequence, the expansion (\ref{expansion:short-range}) is valid
for $\QRC[][\Phi]{n-h}{\chunk{X}{1}{n-h}+\chunk{X}{h+1}{n}}$  and
$\QRC[][\Phi]{n-h}{\chunk{X}{1}{n-h}-\chunk{X}{h+1}{n}}$ with
$a_{n-h}=\sqrt{n-h}$, that is
$$
\sqrt{n-h}
\left[\QRC[][\Phi]{n-h}{\chunk{X}{1}{n-h}\pm\chunk{X}{h+1}{n}}-Q(\Phi_{\sigma,\pm})\right]
=\frac{1}{\sqrt{n-h}}\sum_{i=1}^{n-h}
\IF(X_i\pm X_{i+h},Q,\Phi_{\sigma,\pm})+o_P(1)\; .
$$
Then, applying the Delta method
\citep[Theorem 3.1]{vdv:2000} with the transformation $b(x)=x^2$,
$b'(x)=2x$, we get
\begin{multline*}
\sqrt{n-h}
\left[\QRC[][\Phi]{n-h}{\chunk{X}{1}{n-h}\pm\chunk{X}{h+1}{n}}^2-Q^2(\Phi_{\sigma,\pm})\right]\\
=\frac{2Q(\Phi_{\sigma,\pm})}{\sqrt{n-h}}\sum_{i=1}^{n-h}
\IF(X_i\pm X_{i+h},Q,\Phi_{\sigma,\pm})+o_P(1)\; .
\end{multline*}
Hence $\widehat{\gamma}_Q(h,\chunk{X}{1}{n},\Phi)$ in (\ref{def:gamma_Q})
satisfies the following asymptotic expansion:
\begin{equation}\label{eq:asym_exp_gamma}
\sqrt{n-h}\left(\widehat{\gamma}_Q(h,\chunk{X}{1}{n},\Phi)
-\left\{Q^2(\Phi_{\sigma,+})-Q^2(\Phi_{\sigma,-})\right\}/4\right)
=\frac{1}{\sqrt{n-h}}\sum_{i=1}^{n-h}\psi(X_i,X_{i+h})+o_P(1)\; ,
\end{equation}
where for all $x$ and $y$,
$$
\psi(x,y)=\frac12\left\{Q(\Phi_{\sigma,+})\; \IF\left(x+y,Q,\Phi_{\sigma,+}\right)
 -Q(\Phi_{\sigma,-})\; \IF\left(x-y,Q,\Phi_{\sigma,-}\right)\right\}\; .
$$
Using the identity (\ref{IF:Q_1}), $\psi$ has the expression given in (\ref{def:psi}).
We have now to prove a CLT for
$(n-h)^{-1/2}\sum_{i=1}^{n-h}\psi(X_i,X_{i+h})$.
Using Lemma \ref{lem:herm_rank_psi},
the definition of the Hermite rank given in \cite[p. 2245]{Arcones:1994} and
Assumption (A\ref{assum:short-range}), we obtain that Condition (2.40)
of Theorem 4 \cite[p. 2256]{Arcones:1994} is
satisfied with $\tau=2$. This concludes the proof of the theorem
by observing that $\{Q^2(\Phi_{\sigma,+})-Q^2(\Phi_{\sigma,-})\}/4=\PE[X_1
X_{1+h}]=\gamma(h)$ (see \ref{e:CovQ}).
\end{proof}

\begin{proof}[Proof of Theorem~\ref{theo:Q_n_long-range}]
Since, by scale invariance, $\QRC[][\Phi]{n}{\chunk{X}{1}{n}}-\sigma
= \sigma(\QRC[][\Phi]{n}{\chunk{X}{1}{n}/\sigma}-1)$, we shall
focus in the sequel on the case $\gamma(0)=1$.
First, note that using Lemma \ref{lem:hermite_rank_T1(F_n)}
below the Hermite rank
of the class of functions $\{\1_{\{|\cdot-\cdot|\leq r\}}-U(r)\; ,
r\in [r_0-\eta,r_0+\eta]\}$ is $m=2$, where $U$ is defined in 
(\ref{e:U}) and $r_0$ in (\ref{eq:r_0}).

(i) Suppose first $D>1/2$. Let us verify that the assumptions of
Proposition \ref{theo:U_n_D>1/2} hold.
%
Conditions (\ref{cond:contr_G}) and (\ref{eq:G_lip_2}) are easily verified.
Let us check Condition  (\ref{eq:G_lip}).
Note that for all $\ell\geq 1$, $X_1-X_{1+\ell}\sim\mathcal{N}(0,2(1-\gamma(\ell)))$, thus if
$t\leq s$, there exists a positive constant $C$ such that,
$$
\PE[k(X_1,X_{1+\ell},s)-k(X_1,X_{1+\ell},t)]=\PP(t\leq |X_1-X_{1+\ell}|\leq s)
\leq\frac{2}{\sqrt{4\pi(1-\gamma(\ell))}}|t-s|
\leq C|t-s|\; ,
$$
where $k(x,y,r)=\1_{\{|x-y|\leq r\}}$.
Since $\gamma(\ell)\to 0$ as $\ell\to \infty$, we obtain  (\ref{eq:G_lip}).

Conditions (\ref{cond:h_1}) and (\ref{assum:h_1_Lipschitz}) are satisfied since
\begin{equation}\label{eq:h_1:grassberger}
k_1(x,r)=\PE[\1_{\{|x-Y|\leq r\}}]=\Phi(x+r)-\Phi(x-r)\; .
\end{equation}
Now consider the process
\begin{equation}\label{e:T1Fproc}
\{\sqrt{n}(T_1(F_n)[r]-T_1(F)[r]),r\in[r_0-\eta,r_0+\eta]\}\; ,
\end{equation}
where $F=\Phi$ and
$$
T_1(F)[r]=\int_{\rset^2} \1_{\{|y-x|\leq r\}} \rmd\Phi(x) \rmd\Phi(y)
=\int_{\rset}[\Phi(x+r)-\Phi(x-r)]\rmd\Phi(x)\; .
$$
By Proposition \ref{theo:U_n_D>1/2}, the process (\ref{e:T1Fproc})
converges weakly to a Gaussian process in the space of cadlag functions equipped with the
topology of uniform convergence for some $\eta>0$ when $D>1/2$.

(ii) Suppose now $D<1/2$. Let us check that the assumptions of
Proposition \ref{theo:D<1/2} hold. Condition (\ref{eq:lip_h}) holds since
it is the same as Condition (\ref{eq:G_lip}). Since $k_1$
is a Lipschitz function, so is $U$ defined in (\ref{e:U}). Let us now check
Condition (\ref{eq:lambda_tilde}).
If $s\leq t$
\begin{multline*}
\int_{\rset}\int_{\rset} \1_{\{s<x-y\leq t\}}(|x|+|xy|+|x^2-1|)\phi(x)\phi(y) \rmd x
\rmd y =
\int_{\rset}\left(\int_{x-t}^{x-s}\phi(y) \rmd y\right)|x|\phi(x) \rmd x\\
+\int_{\rset}\left(\int_{x-t}^{x-s}|y|\phi(y) \rmd y\right)|x|\phi(x) \rmd x
+\int_{\rset}\left(\int_{x-t}^{x-s}\phi(y) \rmd y\right)|x^2-1|\phi(x) \rmd x\; .
\end{multline*}
Since $\phi(\cdot)$ and $|.|\phi(\cdot)$ are bounded and that
the moments of Gaussian random variables are all finite, we get (\ref{eq:lambda_tilde}).
Then, applying Proposition \ref{theo:D<1/2} and Lemma
\ref{lem:hermite_rank_T1(F_n)} leads to
the weak convergence of the process
$\{\beta(D)n^{D}/L(n)(T_1(F_n)[r]-T_1(F)[r]),r\in[r_0-\eta,r_0+\eta]\}$
to $\{\dot{\phi}(r/\sqrt{2})(Z_{2,D}(1)-Z_{1,D}(1)^2); r\in [r_0-\eta,r_0+\eta]\}$.

We now want to use the functional Delta method
as in the proof of Lemma \ref{prop:asym_expansion} in both cases (i)
and (ii).

By  \cite[Lemma 21.3]{vdv:2000}, $T_2$ defined in
(\ref{eq:definition-T2}) is Hadamard
differentiable with the following Hadamard differential:
$
DT_2(T_1(\Phi)) \cdot g=-g(r_0)/(T_1(\Phi))'[r_0]\; .
$
Thus $DT_2(T_1(\Phi))$ is a continuous function with respect to
$g$. By the functional Delta method, with $T_0=T_2\circ T_1$, we obtain the following
expansion:
\begin{equation}\label{eq:exp:Q_n:bis}
a_n(\QRC[][\Phi]{n}{\chunk{X}{1}{n}}-Q(\Phi))
=c(\Phi)\;a_n(T_0(F_n)-T_0(\Phi))=-c(\Phi)\;a_n\frac{(T_1(F_n)-T_1(\Phi))[r_0]}{(T_1(\Phi))'[r_0]}+o_P(1)\; ,
\end{equation}
where $a_n=\sqrt{n}$ in the case $(i)$ and $a_n=\beta(D)n^{D}/L(n)$ in the case $(ii)$.
In  case $(i)$,
$$
-c(\Phi)\sqrt{n}\frac{(T_1(F_n)-T_1(\Phi))[r_0]}{(T_1(\Phi))'[r_0]}
\stackrel{d}{\longrightarrow}\mathcal{N}(0,\sigma^2_1)\; ,
$$
where $\sigma^2_1$ is given by Equation (\ref{eq:cov_struct}) in
Proposition \ref{theo:U_n_D>1/2}:
\begin{multline}\label{eq:cas(i)_1}
\sigma^2_1=4\Var\left[-\frac{c(\Phi)}{(T_1(\Phi))'[r_0]}k_1(X_1,r_0)\right]\\
+8\sum_{k\geq 1}\Cov\left[-\frac{c(\Phi)}{(T_1(\Phi))'[r_0]}k_1(X_1,r_0),
-\frac{c(\Phi)}{(T_1(\Phi))'[r_0]}k_1(X_{k+1},r_0)\right]\; ,
\end{multline}
where $k_1$ is defined in (\ref{eq:h_1:grassberger}). 
Since 
$$
\PE[k_1(X_1,r_0)]=\PE[\Phi(X_1+r_0)-\Phi(X_1-r_0)]
=\int_{\rset^2} \1_{\{|y-x|\leq r\}} \rmd\Phi(x) \rmd\Phi(y)=1/4
$$
by (\ref{eq:r_0}) and $r_0=1/c(\Phi)$ by (\ref{e:cr}), we get using
(\ref{e:T1F}) that
\begin{multline}\label{eq:cas(i)_2}
\frac{-2c(\Phi)\left[k_1(X_1,r_0)-\PE(k_1(X_1,r_0))\right]}{(T_1(\Phi))'[r_0]}\\
=c(\Phi)\frac{1/4+\Phi(X_1-1/c(\Phi))-\Phi(X_1+1/c(\Phi))}{2\int_{\rset}\phi(y)\phi(y+1/c(\Phi))
\rmd y}=\IF(X_1,Q,\Phi)\; ,
\end{multline}
where $\IF(\cdot,Q,\Phi)$ is defined in (\ref{IF:Q}).
Using (\ref{eq:cas(i)_1}), (\ref{eq:cas(i)_2}) and (\ref{eq:Q_deg0})
in Lemma \ref{lem:hermite_Q}, we get that
$$
\sigma^2_1=\PE[\IF(X_1,Q,\Phi)]+2\sum_{k\geq 1}
\PE[\IF(X_1,Q,\Phi)\IF(X_{k+1},Q,\Phi)]\; ,
$$
which concludes the proof of $(i)$.
In the case $(ii)$, 
in view of (\ref{eq:exp:Q_n:bis}), it is sufficient to show that
\begin{equation}\label{e:limitD}
-c(\Phi)\beta(D)\frac{n^{D}}{L(n)}\frac{(T_1(F_n)-T_1(\Phi))[r_0]}{(T_1(\Phi))'[r_0]}
\stackrel{d}{\longrightarrow}\frac{1}{2}(Z_{2,D}(1)-Z_{1,D}(1)^2)\; .
\end{equation}
This result follows from the convergence in
distribution of $\beta(D)n^{D}/L(n)(T_1(F_n)[r_0]-T_1(\Phi)[r_0])$
to $\dot{\phi}(r_0/\sqrt{2})(Z_{2,D}(1)-Z_{1,D}(1)^2)$, (\ref{e:T1F}) and
the identity
$
-c(\Phi)\;\dot{\phi}(r_0/\sqrt{2})(2\int_{\rset}\phi(y)\phi(y+r_0)\rmd y)^{-1}
=1/2.
$
This identity follows from 
$\dot{\phi}(r_0/\sqrt{2})=-(2\sqrt{\pi})^{-1}\exp(-r_0^2/4) r_0$
and $r_0=1/c(\Phi)$.
\end{proof}

\begin{proof}[Proof of Proposition \ref{p:sigma2}]

Using the same arguments as those used in the proof of Proposition
\ref{p:sigma}, we get that $\widehat{\sigma}^2_{n,Y}-1$ satisfies
the Hoeffding decomposition (\ref{eq:hoeff_dec_sigma_n}), where $(Y_i)_{i\geq 1}$ satisfies
(A\ref{assum:long-range}) with $\gamma(0)=1$.

(a) If $D > 1/2$, using \cite{dehling:taqqu:1991},
the first term  in the decomposition (\ref{eq:hoeff_dec_sigma_n}) is the
leading one, then using the same arguments as those used in the proof of Proposition
\ref{p:sigma}, we get that  the asymptotic variance of $\sqrt{n} \left(\widehat{\sigma}_{n,X} - \sigma
 \right)$ is equal to
$$
(2\gamma(0))^{-1}(\gamma(0)^2+2\sum_{k\geq 1}\gamma(k)^2) \; .
$$
Using the same upper bound as the one used in the proof of Proposition
\ref{p:sigma}, we get that the relative efficiency of the
robust scale estimator is, in this case, larger than 82.27$\%$.

(b) If $D < 1/2$, we can apply the results of  \cite{dehling:taqqu:1991}
and the classical Delta method to show that
$$
\beta(D) n^D L(n)^{-1}(\widehat{\sigma}_{n,X}-\sigma) \stackrel{d}{\longrightarrow}
\sigma/2 \left(Z_{2,D}(1) - Z^2_{1,D}(1) \right)  \; .
$$
\end{proof}

\begin{proof}[Proof of Theorem~\ref{theo:gamma_long}]
Let $\Phi_{\sigma,+}$ and $\Phi_{\sigma,-}$ denote the c.d.f of
$(X_i+X_{i+h})_{i\geq 1}$ and $(X_i-X_{i+h})_{i\geq 1}$, respectively.
Since
$(X_i)_{i\geq 1}$ satisfies Assumption (A\ref{assum:long-range}), a straightforward
application of a Taylor formula shows that the same holds for
$(X_i+X_{i+h})_{i\geq 1}$ with a scale equal to $Q(\Phi_{\sigma,+})$
and $L$ replaced by some slowly varying function $\widetilde{L}$.
Thus,  in the
case $(i)$, where $D>1/2$, we obtain that 
$\QRC[][\Phi]{n-h}{\{\chunk{X}{1}{n-h}+\chunk{X}{h+1}{n}\}/Q(\Phi_{\sigma,+})}$
satisfies the expansion (\ref{eq:exp:Q_n:bis}) with
$a_{n}=\sqrt{n}$ as proved in the proof of Theorem
\ref{theo:Q_n_long-range}. Using (\ref{eq:dev_UG}), we get that
\begin{multline*}
\sqrt{n}\left\{\QRC[][\Phi]{n-h}{\chunk{X}{1}{n-h}+\chunk{X}{h+1}{n}}-Q(\Phi_{\sigma,+})\right\}\\
=-\frac{c(\Phi)Q(\Phi_{\sigma,+})}
{(T_1(\Phi))'[r_0]}\;\frac{2}{\sqrt{n}}\sum_{i=1}^n 
\left[k_{1}(\{X_i+X_{i+h}\}/Q(\Phi_{\sigma,+}),r_0)-U(r_0)\right]
+o_P(1)\;,
\end{multline*}
where $k_1$ and $U$ are defined in (\ref{eq:h_1:grassberger}) and
(\ref{e:U}), respectively. Thus, using (\ref{eq:cas(i)_2}) and (\ref{IF:Q_1}), we obtain
\begin{equation}\label{eq:dev_QRC_+}
\sqrt{n}\left\{\QRC[][\Phi]{n-h}{\chunk{X}{1}{n-h}+\chunk{X}{h+1}{n}}-Q(\Phi_{\sigma,+})\right\}
=\frac{1}{\sqrt{n}}\sum_{i=1}^n \IF(X_i+X_{i+h},Q,\Phi_{\sigma,+})+o_P(1)\; .
\end{equation} 
In the case $(ii)$, where $D<1/2$, we get from the expansion
(\ref{eq:exp:Q_n:bis}) that
\begin{multline}\label{eq:exp_gamma:D<1/2}
\beta(D)\frac{(n-h)^{D}}{\widetilde{L}(n-h)}(\QRC[][\Phi]{n-h}
{\{\chunk{X}{1}{n-h}+\chunk{X}{h+1}{n}\}/Q(\Phi_{\sigma,+})}-1)\\
=-c(\Phi)
\beta(D)\frac{(n-h)^{D}}{\widetilde{L}(n-h)}\frac{(T_1(F_{+,n-h})-T_1(\Phi))(r_{0})}{(T_1(\Phi))'[r_{0}]}+o_P(1)\; .
\end{multline}
where $F_{+,n-h}$ denotes
the empirical c.d.f of $(\{X_i+X_{i+h}\}/Q(\Phi_{\sigma,+}))_{1\leq i\leq n-h}$.

Let us now focus on the autocovariances and consider first the case
(i) where $D>1/2$. Let us denote by $\gamma_{-}(k)$ the autocovariance
of the process $(X_i-X_{i+h})_{i\geq 1}$ computed at lag $k$. Using a Taylor
formula,
$\gamma_-(k)=O(k^{-2-D+\epsilon})$, for $\epsilon$ in $(0,D)$ such that
$L_i(x)/x^{\epsilon}=O(1)$, as $x$ tends to infinity, for all
$i=0,1,2,3$.
Let $F_{-,n-h}$ denote the empirical c.d.f of $(X_i-X_{i+h})_{1\leq i\leq n-h}$.
Since $\sum_k |\gamma_-(k)|<\infty$,
the process $(X_i-X_{i+h})_{i\geq 1}$ satisfies
Assumption (A\ref{assum:short-range}) implying that
$\sqrt{n}(F_{-,n-h}-\Phi_{\sigma,-})$ converges in distribution
to a Gaussian process in the space of cadlag functions equipped with
the topology of uniform convergence (\cite{csorgo:mielniczuk:1996}).
As a consequence, by Lemma \ref{prop:asym_expansion}, the expansion (\ref{expansion:short-range}) is valid
for $\QRC[][\Phi]{n-h}
{\chunk{X}{1}{n-h}-\chunk{X}{h+1}{n}}$ with
$a_{n}=\sqrt{n}$
where $\IF$ is defined in (\ref{IF:Q}).

Then, in the case $(i)$, using the Delta method
(Theorem 3.1 P. 26 in \cite{vdv:2000}), $\widehat{\gamma}_Q(h,\chunk{X}{1}{n},\Phi)$ satisfies
the following asymptotic expansion as in (\ref{eq:asym_exp_gamma}):
\begin{multline}\label{eq:asym_exp_gamma_bis}
\sqrt{n-h}\left(\widehat{\gamma}_Q(h,\chunk{X}{1}{n},\Phi)
-\left\{Q^2(\Phi_{\sigma,+})-Q^2(\Phi_{\sigma,-})\right\}/4\right)
=\frac{1}{\sqrt{n-h}}\sum_{i=1}^{n-h}\psi(X_i,X_{i+h})+o_P(1)\; ,
\end{multline}
where $\psi$ is defined in (\ref{def:psi}).
Hence, we have to establish a CLT for
$(n-h)^{-1/2}\sum_{i=1}^{n-h}\psi(X_i,X_{i+h})$.
Using Lemma \ref{lem:herm_rank_psi},
the definition of the Hermite rank given on p.~2245 in \cite{Arcones:1994} and
Assumption (A\ref{assum:long-range}) with $D>1/2$, we obtain that
Condition (2.40) of Theorem 4 (P. 2256) in \cite{Arcones:1994} is
satisfied with $\tau=2$. This concludes the proof of $(i)$
by observing from (\ref{e:CovQ}) that $\{Q^2(\Phi_{\sigma,+})-Q^2(\Phi_{\sigma,-})\}/4=\PE[X_1
X_{1+h}]=\gamma(h)$.

Consider the case $(ii)$ where $D<1/2$.
Using (\ref{def:gamma_Q}) and
$\gamma(h)=[Q^2(\Phi_{\sigma,+})-Q^2(\Phi_{\sigma,-})]/4$,
one has
\begin{equation}\label{e:A+A-}
\widehat{\gamma}_Q(h,\chunk{X}{1}{n},\Phi)-\gamma(h)
=A^+_n-A^-_n\; ,
\end{equation}
where
$$
A^{\pm}_n=\frac{1}{4}[\QRC[][\Phi]{n-h}{\chunk{X}{1}{n-h}\pm\chunk{X}{h+1}{n}}^2
-Q^2(\Phi_{\sigma,\pm})]\; .
$$
We first show that the contribution of $A^-_n$ is negligeable. Since
the expansion (\ref{expansion:short-range}) holds for  
$\sqrt{n-h}(\QRC[][\Phi]{n-h}{\chunk{X}{1}{n-h}-\chunk{X}{h+1}{n}}
-Q(\Phi_{\sigma,-}))$, we conclude by arguing as in the proof of
Theorem \ref{theo:short-range}, that this expression is $O_P(1)$.
Applying  the Delta method, we get the same type of result for
$Q^2_{n-h}$, namely
$
\sqrt{n-h}(\QRC[][\Phi]{n-h}{\chunk{X}{1}{n-h}-\chunk{X}{h+1}{n}}^2
-Q^2(\Phi_{\sigma,-}))=O_P(1)
$
and therefore, since $D<1/2$,
\begin{equation}\label{e:A-}
\beta(D)\frac{(n-h)^{D-1/2}}{\widetilde{L}(n-h)}
\sqrt{n-h}\; A^-_n=o_P(1)\; .
\end{equation}
We now turn to $A^+_n$. Applying the Delta method with the
transformation $b(x)=x^2$ to (\ref{eq:exp_gamma:D<1/2}) and using (\ref{e:limitD}) yields
\begin{multline*}
\beta(D)\frac{(n-h)^{D}}{\widetilde{L}(n-h)} A^+_n
=-\frac{c(\Phi)\beta(D)}{2}\frac{(n-h)^{D}}{\widetilde{L}(n-h)}
Q^2(\Phi_{\sigma,+})\frac{(T_1(F_{+,n-h})-T_1(\Phi))[r_0]}{(T_1(\Phi))'[r_0]}+o_P(1)\\
\stackrel{d}{\longrightarrow}\frac{Q^2(\Phi_{\sigma,+})}{4}(Z_{2,D}(1)-Z_{1,D}(1)^2)\; .
\end{multline*}
The result follows from (\ref{e:A+A-}), (\ref{e:A-}) and 
$Q^2(\Phi_{\sigma,+})=\Var(X_1+X_h)=2(\gamma(0)+\gamma(h)).$

\end{proof}

\begin{proof}[Proof of Proposition \ref{p:sigma3}]

The classical autocovariance estimator can be obtained 
from the classical scale estimator $\widehat{\sigma}_{n,X}$ as in
Equation (\ref{def:gamma_Q_init}).  More precisely, a 
straightforward calculation leads to
\begin{equation}\label{eq:classi_auto}
\widehat{\gamma}(h)=\frac{1}{4}\left(\widehat{\sigma}_{n-h,X_{1:n-h}+X_{h+1:n}}^2
-\widehat{\sigma}_{n-h,X_{1:n-h}-X_{h+1:n}}^2\right)(1+o(1))+O_P\left(\frac{1}{n^2}\right)\;.
\end{equation}
In order to alleviate the notations, 
$\widehat{\sigma}_{n-h,X_{1:n-h}+X_{h+1:n}}$ will now be denoted by 
$\widehat{\sigma}_+$ and $\widehat{\sigma}_{n-h,X_{1:n-h}-X_{h+1:n}}$ 
by $\widehat{\sigma}_-$. 

On the one hand, using Proposition \ref{p:sigma2} and the 
same arguments as in the beginning of the proof of Theorem 
\ref{theo:gamma_long}, we have
$$
\beta(D)\frac{n^{D}}{\widetilde{L}(n)}\left(\widehat{\sigma}_+-\sigma_+\right)
\stackrel{d}{\longrightarrow}\frac{\sigma_+}{2}(Z_{2,D}(1)-Z_{1,D}(1)^2)\;,$$
where $\sigma_+$ denotes the standard deviation of $X_1+X_{1+h}$ and
$\widetilde{L}(n)$ is defined in Theorem \ref{theo:gamma_long}. 
Note that $\sigma_+^2=2(\gamma(0)+\gamma(h))$. 
By the classical Delta method, we thus obtain
\begin{equation}\label{eq:sigma_plus}
\beta(D)\frac{n^{D}}{\widetilde{L}(n)}\left(\widehat{\sigma}^2_+-\sigma^2_+\right)
\stackrel{d}{\longrightarrow}2(\gamma(0)+\gamma(h))(Z_{2,D}(1)-Z_{1,D}(1)^2).
\end{equation}
On the other hand, by the same arguments as in Theorem
\ref{theo:gamma_long}, the process $(X_i-X_{i+h})_{i\geq 1}$ 
satisfies Assumption (A\ref{assum:short-range}). 
Let $\sigma_-^2=2(\gamma(0)-\gamma(h))$ denote the variance of
$X_1-X_{1+h}$.
Then as in the proof of Proposition \ref{p:sigma}, 
$ \sqrt{n}\left(\widehat{\sigma}^2_--\sigma^2_-\right)$ 
converges in distribution. This implies that 
\begin{equation}\label{eq:sigma_moins}
\beta(D)\frac{n^{D}}{\widetilde{L}(n)}\left(\widehat{\sigma}^2_--\sigma^2_-\right)=o_P(1).
\end{equation}
Using (\ref{eq:classi_auto}), (\ref{eq:sigma_plus}) and (\ref{eq:sigma_moins}), we get: 
$$\beta(D)\frac{n^{D}}{\widetilde{L}(n)}\left(\widehat{\gamma}(h)-\gamma(h)\right)
\stackrel{d}{\longrightarrow}\frac{\gamma(0)+\gamma(h)}{2}(Z_{2,D}(1)-Z_{1,D}(1)^2)\;.$$

\end{proof}

\section{Technical lemmas}\label{sec:lemmas}

\begin{lemma}\label{lem:hermite_Q}
Let $X$ be a standard Gaussian random variable. The influence function
defined in \eqref{IF:Q} has the
following properties:
\begin{equation}\label{eq:Q_deg0}
\PE[\IF(X,Q,\Phi)]=0\; ,
\end{equation}
\begin{equation}\label{eq:Q_deg1}
\PE[X \; \IF(X,Q,\Phi)]=0\; ,
\end{equation}
\begin{equation}\label{eq:Q_deg2}
\PE[X^2 \; \IF(X,Q,\Phi)]=(2\sqrt{\pi}\beta)^{-1}\exp(-1/(4c^2)) \neq 0\; ,
\end{equation}
where $\Phi$ is the c.d.f of a standard Gaussian random variable, $c=c(\Phi)$
is defined in (\ref{eq:definition-Qn}) and $\beta=\int \phi(y)\phi(y+1/c) \rmd y$.
\end{lemma}

\begin{proof}[Proof of Lemma \ref{lem:hermite_Q}]
Let us first prove that $\PE[\IF(X,Q,\Phi)]=0$. It is enough to prove that
$\PE[\Phi(X+1/c)-\Phi(X-1/c)]=1/4$. Using the definition of $c$,
namely (\ref{e:e2}) or (\ref{eq:r_0}), we get:
\begin{multline}\label{e:cc}
\PE[\Phi(X+1/c)-\Phi(X-1/c)]
=\int_{\rset}(\Phi(x+1/c)-\Phi(x-1/c))\phi(x) \rmd x\\
=\int_{\rset^2}\1_{\{|y-x|\leq 1/c\}}\phi(x)\phi(y) \rmd x \rmd y
=T_1(\Phi)[1/c]=1/4\; .
\end{multline}
Then, let us prove that $\PE[X \IF(X,Q,\Phi)]=0$. Since $X$ has a standard Gaussian distribution,
it suffices to prove that $\PE[X\{\Phi(X+1/c)-\Phi(X-1/c)\}]=0$. By symmetry of $\phi$, we obtain:
\begin{multline*}
\PE[X\Phi(X+1/c)]=\int_{\rset}x\Phi(x+1/c)\phi(x) \rmd x
=\int_{\rset}x(1-\Phi(-x-1/c))\phi(x) \rmd x \\
= -\int_{\rset}x\Phi(-x-1/c) \phi(x) \rmd x = \PE[X\Phi(X-1/c)]\; .
\end{multline*}
Finally, let us compute: $\PE[X^2 \IF(X,Q,\Phi)]$. Set $\beta=\int \phi(y)\phi(y+1/c) \rmd y$.
By integrating by parts, using (\ref{e:cc}) and finally the symmetry
of $\phi$, we get
\begin{multline*}
(\beta/c)\PE[X^2 \IF(X,Q,\Phi)]=-\int_{\rset}\left(\int_{y-1/c}^{y+1/c} x^2 \phi(x) \rmd x\right)\phi(y) \rmd y +1/4\\
=-\int_{\rset}\left\{(y-1/c)\phi(y-1/c)-(y+1/c)\phi(y+1/c)\right\}\phi(y) \rmd y
-\int_{\rset}\left(\int_{y-1/c}^{y+1/c}\phi(x)\rmd x\right)\phi(y)\rmd y+1/4\\
=\int_{\rset}\left\{-(y-1/c)\phi(y-1/c)+(y+1/c)\phi(y+1/c)\right\}\phi(y) \rmd y\; ,
\end{multline*}
where the last equality comes from 
$\int_{\rset}\left(\int_{y-1/c}^{y+1/c}\phi(x)\rmd x\right)\phi(y)\rmd y=T_1(\Phi)(1/c)=1/4.$
By symmetry of $\phi$,
\begin{multline*}
\int_{\rset}\left\{-(y-1/c)\phi(y-1/c)+(y+1/c)\phi(y+1/c)\right\}\phi(y) \rmd y
=-2  \int_{\rset}(y-1/c)\phi(y-1/c)\phi(y) \rmd y\\
=(2c\sqrt{\pi})^{-1}\exp(-1/(4c^2)) \; ,
\end{multline*}
which concludes the proof.
\end{proof}

\begin{lemma}\label{lem:herm_rank_psi}
Let $(X,Y)$ be a standard Gaussian random vector such that
$\Cov(X,Y)=0$ and let $\Phi_{+}$ and $\Phi_{-}$
denote the c.d.f. of $X+Y$ and $X-Y$, respectively.
The influence function $\psi$ defined, for all $x$ and $y$ in $\rset$, by
$$
\psi(x,y)=\frac12\left\{Q(\Phi_{+})\; \IF\left(x+y,Q,\Phi_{+}\right)
 -Q(\Phi_{-})\; \IF\left(x-y,Q,\Phi_{-}\right)\right\}\; ,
$$
satisfies the following properties:
\begin{equation}\label{eq:psi_deg0}
\PE[\psi(X,Y)]=0\; ,
\end{equation}
\begin{equation}\label{eq:psi_deg1}
\PE[X\psi(X,Y)]=\PE[Y\psi(X,Y)]=0\; ,
\end{equation}
\begin{equation}\label{eq:psi_deg2}
\PE[XY\psi(X,Y)]\neq 0\; .
\end{equation}
\end{lemma}

\begin{proof}[Proof of Lemma \ref{lem:herm_rank_psi}]
Using (\ref{IF:Q_1}), (\ref{eq:Q_deg0}) and 
$Q(\Phi_{\pm})^2=\Var(X\pm Y)$ (see (\ref{e:VarQ})), we get that
$$
\PE[\psi(X,Y)]=\frac12\{Q(\Phi_+)^2-Q(\Phi_-)^2\}\PE\left[\IF(U,Q,\Phi)\right]=0\; ,
$$
where $U$ is a standard Gaussian random variable, which gives (\ref{eq:psi_deg0}).
Let us now prove  (\ref{eq:psi_deg1}). First note that,
$$
\PE[X\psi(X,Y)]=\frac12\left\{\PE[(X+Y)\psi(X,Y)]+\PE[(X-Y)\psi(X,Y)]\right\}\; .
$$
But,
\begin{eqnarray*}
\PE[(X+Y)\psi(X,Y)]&=&\frac12\PE\left[Q(\Phi_+)^2 (X+Y)\IF((X+Y)/Q(\Phi_+),Q,\Phi)\right.\\
&-&\left. Q(\Phi_-)^2 (X+Y)\IF((X-Y)/Q(\Phi_-),Q,\Phi)\right]\\
&=&\frac12\left[Q(\Phi_+)^3 \PE[U\IF(U,Q,\Phi)
-Q(\Phi_-)^2 Q(\Phi_+)\PE[U\IF(V,Q,\Phi)\right]\; ,
\end{eqnarray*}
where $U=(X+Y)/Q(\Phi_+)$ and $V=(X-Y)/Q(\Phi_-)$
are independent standard Gaussian random variables. By
(\ref{eq:Q_deg1}), $\PE[(X+Y)\psi(X,Y)]=0$. In the same way,
 $\PE[(X-Y)\psi(X,Y)]=0$ which gives (\ref{eq:psi_deg1}).
Let us now prove (\ref{eq:psi_deg2}).
Using that $4XY=(X+Y)^2-(X-Y)^2$, we get
\begin{multline}\label{eq:psi_deg2_detail}
8\PE[XY\psi(X,Y)]
=\PE[(X+Y)^2 Q(\Phi_+) \IF(X+Y,Q,\Phi_+)+(X-Y)^2 Q(\Phi_-) \IF(X-Y,Q,\Phi_-)]\\
-\PE[(X-Y)^2 Q(\Phi_+) \IF(X+Y,Q,\Phi_+)+(X+Y)^2 Q(\Phi_-)
\IF(X-Y,Q,\Phi_-)]\\
=(Q(\Phi_+)^4+Q(\Phi_-)^4) \PE[U^2 \IF(U,Q,\Phi)]
-Q(\Phi_+)^2 Q(\Phi_-)^2 \left(\PE[V^2 \IF(U,Q,\Phi)]+\PE[U^2 \IF(V,Q,\Phi)]\right)\; ,
\end{multline}
where $U$ and $V$ are as above. The first term is non-zero by
(\ref{eq:Q_deg2}) while the second term is zero by independence of 
$U$ and $V$ and (\ref{eq:Q_deg0}). This yields (\ref{eq:psi_deg2}).

\end{proof}

\begin{lemma}\label{lem:hermite_rank_T1(F_n)}
Let $\alpha_{p,q}(r)=\PE[\1_{\{|X-Y|\leq r\}}H_p(X)H_p(Y)]$ where $X$ and $Y$
are independent standard Gaussian random variables and $H_p$ is the $p$th Hermite polynomial
with leading coefficient equal to 1. Then,
\begin{enumerate}
\item [(i)] $\alpha_{1,0}(r)=0,\; \forall r\in\rset$
\item [(ii)] $\alpha_{2,0}(r)=-\alpha_{1,1}(r)=\dot{\phi}(r/\sqrt{2}),\; \forall
r\in\rset$
\item[(iii)] Moreover, there exists some positive $\eta$ 
  such as $\alpha_{2,0}(r)=-\alpha_{1,1}(r)$
is different from 0 when $r$ is in $ [r_0-\eta;r_0+\eta]$,
where  $r_0$ is defined in (\ref{eq:c(Phi)}).
\end{enumerate}
\end{lemma}

\begin{proof}[Proof of Lemma \ref{lem:hermite_rank_T1(F_n)}]

The proof of (i) follows from the symmetry of the Gaussian
distribution and the proof of (ii) relies on
the following identity: for all $r\in\rset$,
$$\int_\rset (\phi(x+r)-\phi(x-r))x\phi(x)\rmd x
=\dot{\phi}(r/\sqrt{2}).$$

Let us now turn to the proof of (iii). $\dot{\phi}(r/\sqrt{2})$ is
equal to zero only if $r=0$.
By (\ref{eq:c(Phi)}), $r_0$ is such that $\Phi(r_0/\sqrt{2})=5/8$, and hence
is different from 0. The existence of $\eta$
follows from the continuity of $\dot{\phi}$.
\end{proof}

\section{Conclusion}\label{sec:conc}

In this paper, we studied 
the asymptotic properties of the robust scale estimator $\textrm{Q}_n$
(\cite{Rousseeuw:Croux:1993}) and of the robust autocovariance estimator
$\widehat{\gamma}_Q$ (\cite{Genton:Ma:2000}), for short and
long-range dependent processes. We showed that the
asymptotic variance of these estimators is optimal, or close to it,
and we verified, by using simulations, that these estimators are indeed
robust in the presence of outliers.
Complete proofs of the asymptotic properties of the robust scale
$\textrm{Q}_n$ and the covariance estimator $\widehat{\gamma}_Q$ are provided for Gaussian stationary processes.
The central limit theorems for $\textrm{Q}_n$ and $\widehat{\gamma}_Q$ were
obtained. In all cases, the rate of convergence
of the estimators is $\sqrt{n}$, except for long-range dependent
processes with $D \in (0,1/2)$, for which the rate is $n^{D}L(n)^{-1}$.
Empirical Monte-Carlo experiments were conducted in order to
illustrate the finite sample size properties of the estimators.
The robustness of $\textrm{Q}_n$ and $\widehat{\gamma}_Q$ were also investigated
when the process contained outliers. The theoretical results and the
empirical evidence strongly suggest
the use of these estimators as an alternative to estimate the scale and
the autocovariance structure of the process. The classical scale 
and autocovariance  estimators  were also considered as means of
comparison. All estimators showed similar empirical accuracy when
the data did not contain outliers. However, the classical estimators were
significantly affected when additive outliers are present.
The robust ones, however, were much less affected.

\bibliographystyle{chicago}
\bibliography{robust_cov_biblio}

\end{document}